\documentclass{amsart}

\usepackage[utf8x]{inputenc}

\usepackage{graphicx}
\usepackage[shortlabels]{enumitem}
\usepackage{verbatim}
\usepackage{amssymb}
\usepackage{mathrsfs}  
\usepackage{mathabx}
\usepackage{nicefrac}
\usepackage{stmaryrd}

\usepackage{color}

\usepackage{tikz}
\usetikzlibrary{arrows,chains,matrix,positioning,scopes}
\usetikzlibrary{arrows.meta}
\usetikzlibrary{matrix}

\usepackage{hyperref}
\usepackage {cleveref}

\DeclareMathOperator{\frob}{Frob}
\DeclareMathOperator{\ord}{ord}

\newtheorem{theorem}{Theorem}[section]
\newtheorem{proposition}[theorem]{Proposition}
\newtheorem{lemma}[theorem]{Lemma}
\newtheorem{corollary}[theorem]{Corollary}

\theoremstyle{definition}
\newtheorem{definition}[theorem]{Definition}

\theoremstyle{remark}
\newtheorem{remark}[theorem]{Remark}
\newtheorem{example}[theorem]{Example}

\numberwithin{equation}{section}

\begin{document}

\title{Wolstenholme Type Congruences and Framing of Rational $2$-Functions}

\author[L.F. M\"{u}ller]{L. Felipe M\"{u}ller}
\address{Mathematisches Institut, Universit\"{a}t Heidelberg, Im Neuenheimer Feld 205, 69120 Heidelberg, Germany}
\email[corresponding author]{lmueller@mathi.uni-heidelberg.de}


\date{\today.}


\begin{abstract}
We show that the framing of $2$-sequences whose generating functions are rational integrate to $3$-sequences.
To do so, we give a generalization of Wolstenholme's Theorem.
\end{abstract}

\maketitle
\tableofcontents

\section{Introduction}

Sequences of integers $(a_n)_{n\in\mathbb{N}} \in \mathbb{Z}$ satisfying
\begin{align}\label{supercongruence}
	a_{mp^r}\equiv a_{mp^{r-1}} \mod p^{sr},
\end{align}
for all $m,r\in \mathbb{N}$ and a fixed $s\in \mathbb{N}$, are referred to as \textit{$s$-realizable  sequences} in \cite{alm06}.
For instance, the sequence of coefficients of the Maclaurin expansion of the Yukawa coupling is expected to be $3$-realizable (here, $s=3$).
Taking the Lambert expansion of the generating power series of an $s$-realizable sequence  $(a_n)_{n \in \mathbb{N}}$
\begin{align}\label{eq: lambert formal extension}
\sum_{n=1}^\infty a_n z^n = \sum_{n=1}^\infty b_n n^s \frac{z^n}{1-z^n}
\end{align}
gives integral coefficients $(b_n)_{n \in \mathbb{N}}$ (and vice versa, given integers $(b_n)_{n \in \mathbb{N}}$ in \cref{eq: lambert formal extension}, one obtains an $s$-realizable sequence $(a_n)_{n \in \mathbb{N}}$).
In the case of the Yukawa coupling when the moduli space of complex structures is one-dimensional, the coefficients $\{b_n\}_n$ are realized by instanton numbers, \cite{kon06}.
Indeed, according to the Mirror Symmetry Conjecture (see \cite{can92}, \cite{mor93}) the number $b_n$ in the case of the Yukawa coupling is the number of rational curves of degree $n$ on a generic quintic hypersurface in projective space $\mathbb{P}^4$.
In particular, Mirror Symmetry predicts the numbers $b_n$, $n\in \mathbb{N}$, to be integers, which is a highly non-trivial fact and which is equivalent to $(a_n)_{n\in \mathbb{N}}$ being $3$-realizable.

Let $K$ be an algebraic number field and $\mathcal{O}$ its ring of algebraic integers and $D$ its discriminant.
Following the conventions given in \cite{mue20}, we consider a generalization of $s$-realizable sequences to sequences of algebraic integers in $K$, called $s$-sequences.
For $s\in \mathbb{N}$, an $s$-sequence is a sequence $(a_n)\in K^\mathbb{N}$, such that for any unramified prime ideal $\mathfrak{p}\in \mathcal{O}$ lying above the prime $p\in \mathbb{Z}$, $a_n\in \mathcal{O}_\mathfrak{p}$, and for all $m, r\in \mathbb{N}$,
\begin{align}\label{eq: s-sequence intro}
\frob_\mathfrak{p} \left( a_{p^{r-1}m} \right) - a_{p^rm} \equiv 0 \mod \mathfrak{p}^{sr}\mathcal{O}_\mathfrak{p},
\end{align}
where $\mathcal{O}_\mathfrak{p}$ is the ring of $\mathfrak{p}$-adic integers and $\frob_{\mathfrak{p}}$ is the canonical lift of the standard Frobenius element of $\mathfrak{p}$ in the Galois group of the local field extension $(\mathcal{O}/ \mathfrak{p}) | (\mathbb{Z}/p)$.

One of the most interesting observations concerning $2$-sequences in particular is that the generating functions of these sequences permit a certain algebraic transformation (of formal power series) called \textit{framing}.
Formally, framing to the parameter $\nu \in \mathbb{Z}$ can be characterized by a functional equation.
Let $V \in z K\llbracket z \rrbracket$ be a power series then the $\nu$-framing $V^{(-,\nu)} \in z K \llbracket z \rrbracket$ (the minus sign in the index referes to a sign convention explained later) of $V$  gives a power series satisfying the functional equation
\begin{align} \label{eq: intro framing op phi-}
\smallint V^{(-,\nu)} \left( z\left( - \exp(- \smallint V(z))^\nu \right) \right) = \smallint V(z),
\end{align}
where $\smallint$ is defined on power series by $\smallint\colon z^n \mapsto \frac{z^n}n$ for $n\in \mathbb{N}$.

These framing transformations appear in the context of open topological string theory, see \cite{aga02}, where the name has been coined.
In \cite{oog}, it was expected that for an appropriate choice of parametrization, the coefficients of the Lambert expansion \cref{eq: lambert formal extension}, the coefficients $(b_n)$ are counting dimensions of spaces of appropriate BPS states, hence  $(b_n)_n \in \mathbb{Z}^\mathbb{N}$.
In this setting, the superpotential (and its BPS invariants) depend on the integer parameter $\nu \in \mathbb{Z}$, called “\textit{the framing}”.
Framing therefore results from an ambiguity in the identification of the open string modulus.
The main result in \cite{walcher16}, due to Schwarz, Vologodsky and Walcher, is the \textit{Integrality of Framing Theorem} that states that the framing operator preserves $2$-functions (also for more general algebraic coefficients) and defines a group action of $\mathbb{Z}$ on the set of $2$-functions.

There seems to be a subclass of $2$-functions whose framings integrate to $3$-functions.
For instance, this behavior has been observed in \cite{gar15} by the (extremal) BPS invariants of twist knots and has been referred therein as an “improved integrality”.
These $3$-functions appear as solutions of so-called extremal A-polynomials of these knots, and all their framings are expected (\cite[Conj. 1.3]{gar15}) to be also $3$-functions, or at least, can be lifted to $3$-functions by multiplying with an appropriate constant.
This improved integrality does not hold in general.
However, it would be interesting to give a physical interpretation of this property.

In the present paper, the author tries to identify the subclass of those $2$-sequences, such that all framings of the corresponding generating functions integrate to $3$-functions (i.e. the sequences of coefficients of their Maclaurin expansion are $3$-sequences).
The result is given below by \Cref{thm: thm 4.34}.
Let $ V^{(+, \nu)}(z):= V^{(-,\nu)}((-1)^\nu z) \in z K\llbracket z \rrbracket$.
This sign convention for $V^{(+,\nu)}$ is for simplifying some notations and calculations, since it does not effect the congruence condition \cref{eq: s-sequence intro} for the coefficients of $V^{(+,\nu)}$, except for $p=2$.
Recall from \cite{mue20}, that $\mathcal{S}^s(K|\mathbb{Q})$ denotes the set of generating functions of $s$-sequences, while an element in $\overline{\mathcal{S}}^s(K|\mathbb{Q})_\mathrm{fin}$ is a powers series, for which the underlying sequence is  a $\mathbb{Z}$-linear multiple of a sequence that satisfies \cref{eq: s-sequence intro} for almost all unramified prime ideals in $K$.
Also, $\Phi^{\pm}\colon \mathbb{Q} \times z K\llbracket z \rrbracket \rightarrow z K \llbracket z \rrbracket$ is given by $\Phi^{\pm} (\nu,V) = V^{(\pm, \nu)}(z) $.
We have

\begin{theorem}\label{thm: thm 4.34}
Then
\begin{align*}
\Phi^+\left(\mathbb{Z}\left[ D^{-1} \right] \times \mathcal{S}^2_\mathrm{rat}( K|\mathbb{Q} )\right) &\subset \overline{\mathcal{S}}^3( K | \mathbb{Q} )_\mathrm{fin},
\quad \text{and}\\
\Phi^-\left(\mathbb{Z} \times \mathcal{S}^2_\mathrm{rat}( K|\mathbb{Q} )\right) &\subset \overline{\mathcal{S}}^3( K | \mathbb{Q} )_\mathrm{fin}.
\end{align*}
More precisely, let $V\in \mathcal{S}_{\mathrm{rat}}^2(K|\mathbb{Q})$ be a rational $2$-function of periodicity $N \in \mathbb{N}$ and $\nu \in \mathbb{Z}\left[ D^{-1} \right]$ and let $a_n^+= \left[ V^{(+,\nu)}(z) \right]_n$ denote the $n$-th coefficient of $V^{(+,\nu)}(z)$ for all $n\in \mathbb{N}$. Then, for all primes $p$ which are unramified in $K|\mathbb{Q}$ such that $p\nmid N$ and all $n\in \mathbb{N}$ -- except for the case where $p=2$ and $\ord_2(n)=0$ -- we have
\begin{align*}
\mathrm{Frob}_\mathfrak{p}(a_n^+)- a_{pn}^+ 
\equiv
0 \mod p^{2(\ord_p(n)+1) - \delta_{2,p} + \max\{ 0 , \ord_p(n)+1 - \gamma_p \}} \mathcal{O}_\mathfrak{p},
\end{align*}
where $\gamma_p$ is given by
\begin{align*}
\gamma_p = 
\begin{cases}
 1+ \ord_2(N+1),& \text{if $p=2$ and $2\nmid N$},\\
1,& \text{if $p=3$},\\
0,& \text{if $p\geq 5$}.
\end{cases}
\end{align*}
In particular, for all primes $p\geq 5$, which are unramified in $K|\mathbb{Q}$ and does not divide $N$, we find for all $r,m\in \mathbb{N}$,
\begin{align*}
\mathrm{Frob}_\mathfrak{p}(a_{mp^{r-1}}^+)- a_{mp^r}^+ \equiv 0 \mod p^{3r} \mathcal{O}_\mathfrak{p}.
\end{align*}
\end{theorem}

In \cite{mue20}, the author gave a description of the elements in $\mathcal{S}_\mathrm{rat}^2 (K|\mathbb{Q})$, which contributes to the proof of  \Cref{thm: thm 4.34}.
The main result in \cite{mue20} states that the coefficients of an element $V \in \mathcal{S}_\mathrm{rat}^2(K|\mathbb{Q})$ are periodic and lie in a cyclotomic field.
Let us give a summary of the proof for \Cref{thm: thm 4.34} for $\nu = 1$.
From the Integrality of Framing Theorem we directly obtain that
\begin{align}\label{eq: intro further estimation}
\frac2{p^2n^2}\cdot \left( \frob_\mathfrak{p} \left( a^+_n \right) - a^+_{pn} \right),
\end{align}
is an $p$-adic integer, that is a $\mathfrak{p}$-adic integer for all $\mathfrak{p} \mid (p)$ in $\mathcal{O}$.
Assuming rationality of $V$ we then find for all but finitely many $p${ \small
\begin{align}\label{eq: intro explicitly crucial reduction} 
\frac2{p^2n^2}\cdot \left( \frob_\mathfrak{p} \left( a^+_n \right) - a^+_{pn} \right)&\equiv \sum_{m=0}^n \left( \tilde y_{n,m,p} \sum_{\substack{\ell = 1 \\ p\nmid \ell}}^{p(n-m)} \frac{a_{p(n-m)-\ell} a_\ell}{\ell^2} \right) \mod p^{\ord_p(pn)-\delta_{3,p}} \mathcal{O}_\mathfrak{p},
\end{align}}
where $\tilde y_{n,m,p}$ are certain $p$-adic integers in $\mathcal{O}$.
The sum appearing in the big bracket of \cref{eq: intro explicitly crucial reduction}, namely
\begin{align}\label{eq: intro wolstenholmes contribution}
\sum_{\substack{\ell = 1 \\ p\nmid \ell}}^{p(n-m)} \frac{a_{p(n-m)-\ell} a_\ell}{\ell^2},
\end{align}
should be interpreted as a weighted harmonic sum, weighted by a convolution of the $2$-sequence $(a_n)_n$ with itself.
At the same time, many supercongruences are known among the binomial coefficients, for instance the Jacobsthal-Kazandzidis, which are typically proven by using sharp $p$-adic estimations of harmonic sums.
This kind of $p$-adic estimations are often referred to as Wolstenholme type congruences, see also Wolstenholme's Theorem (for instance \cite{mes11}).
This connection led to the following generalization of Wolstenholme's Theorem, handling the sum \eqref{eq: intro wolstenholmes contribution}.

\begin{theorem}\label{lemma: intro general wolstenholme}
Let $p$ be an unramified prime in $K|\mathbb{Q}$ and $\mathfrak{p}\subset \mathcal{O}$ be a prime ideal dividing $(p)$.
Let $(a_k)_{k\in \mathbb{N}} \in \mathcal{O}_\mathfrak{p}^\mathbb{N}$  be a periodic sequence of periodicity $N$, i.e. $N\in \mathbb{N}$ is given by
\begin{align*}
N= \min \{ i\in \mathbb{N} \,|\, a_{k + i} = a_k \text{ for all $n\in \mathbb{N}$} \}.
\end{align*}
Then, for all $n \in \mathbb{N}$,
\begin{align*}
\sum_{\substack{k =1 \\ p\nmid k}}^n \frac{a_{n-k}a_k}{k^2}\equiv 0 \mod p^{ \max \{ 0 , \ord_p(n) - \varepsilon_{p,N} \}}\mathcal{O}_\mathfrak{p},
\end{align*}
where
\begin{align*}
\varepsilon_{p,N}= 
\begin{cases}
\max \{ \ord_2(N) , \ord_2(N+2) \} ,& \text{if $p=2$ and $2 \mid N$},\\
1+\ord_2(N+1),& \text{if $p=2$ and $2\nmid N$},\\
1+ \ord_3(N),&\text{if $p=3$,}\\
\ord_p(N),& \text{if $p\geq 5$}.
\end{cases}
\end{align*}
\end{theorem}

Finally, $\tilde y_{n,m,p}$ sharply contributes the remaining $p$-divisibility to obtain \Cref{thm: thm 4.34}.
This contribution can be considered as an auxiliary to Dwork's Integrality Lemma, see  \cite[Lem. 1]{dwo58} for the original result.
In the setting of $s$-functions it is given by the following statement.
Let $V \in z K \llbracket z \rrbracket$ and let $Y \in 1 + zK\llbracket z \rrbracket$ be related by $Y(z)= \exp( \smallint V(z))$.
Then $V$ is the generating series of an $1$-sequence if and only if $Y$ has integral coefficients at all unramified prime ideals $\mathfrak{p}\subset \mathcal{O}$.
Dwork himself used his lemma (stated for $K=\mathbb{Q}$) as a key step to prove his theorem that for an affine hypersurface $H$ over a finite field $\mathbb{F}_q$ the zeta-function $Z(H;X)$ of $H$ in the variable $X$ is a rational function and its logarithmic derivative $\frac{Z'(H; X)}{Z(H;X)}$ is the generating function of the non-negative numbers $(N_n)_{n\in\mathbb{N}}$ of $\mathbb{F}_{q^n}$-points of $H$, i.e. $N_n = |H(\mathbb{F}_{q^n})|$.
We need the following statement

\begin{proposition}\label{lemma: intro final countdown}
Let $V \in \mathcal{S}^1(K|\mathbb{Q})$ be the generating series of an $1$-sequence and let $p$ be an unramified prime in $K|\mathbb{Q}$.
Then for all $n,m\in \mathbb{N}$ with $\ord_p(n)\geq \ord_p(m)$, we find the following $\mathfrak{p}$-adic estimation for the $m$-th coefficient $\tilde y_m$ of the function $\exp\left(n \smallint V(z)\right)$
\begin{align}\label{eq: intro 3-framing integrality}
\tilde y_m \equiv 0 \mod p^{\ord_p(n)-\ord_p(m)} \mathcal{O}_\mathfrak{p}.
\end{align}
\end{proposition}

\Cref{thm: thm 4.34} can be considered as a generalization of the Jacobsthal-Kazandzidis congruence.
Also the proof of the latter served as a source of inspiration to the author in the process of  finding the proof of  \Cref{thm: thm 4.34}.
However, the Jacobsthal-Kazandzidis congruence does not follow from \Cref{thm: thm 4.34}, yet!
For this we give an extension of \Cref{thm: thm 4.34} and also of the Integrality of Framing Theorem, where we allow $\nu$ to be a rational number in the following manner.
For a power series $\sum_{n=0}^\infty x_n z^n \in K \llbracket z \rrbracket$ and an number $\ell \in \mathbb{Z}$ the \textit{Cartier operator} $\mathscr{C}_\ell$ is given by
\begin{align*}
\mathscr{C}_\ell \left( \sum_{n=0}^\infty x_n z^n \right) = \sum_{n=0}^\infty x_{\ell n} z^n.
\end{align*}
Then fractional framing referes to power series obtained by $ \mathscr{C}_\sigma V^{(\pm, \nu)}$, where $V\in z K\llbracket z \rrbracket$, $\nu\in \mathbb{Q}$ and $\sigma \in \mathbb{N}$.

\begin{theorem}\label{thm: intro 2-integrality of fractional framing}
Let $\sigma \in \mathbb{N}$.
\begin{enumerate}[$(1)$]
\item{Integrality of Fractional Framing:}
Then,
\begin{align*}
\left( \frac1\sigma  \mathscr{C}_\sigma \circ  \Phi^- \right) \left( \left( \frac1\sigma  \mathbb{Z} \right)\times \mathcal{S}^2(K|\mathbb{Q}) \right) \subset \mathcal{S}^2 (K | \mathbb{Q} )
\end{align*}
and
\begin{align*}
\left( \frac1\sigma \mathscr{C}_\sigma \circ \Phi^+ \right) \left( \left( \frac1\sigma\mathbb{Z} \right) \times \mathcal{S}^2_\mathrm{rat}(K|\mathbb{Q})\right) \subset \left( \mathcal{S}^2 (K | \mathbb{Q} )_{\{2\}} \cap \overline{\mathcal{S}}^2(K|\mathbb{Q})\right).
\end{align*}
\item{Improved Integrality of Fractional Framing:}
Then
\begin{align*}
\left( \frac1\sigma \mathscr{C}_\sigma \circ \Phi^{+/-} \right) \left( \left( \frac1\sigma\mathbb{Z} \right) \times \mathcal{S}^2_\mathrm{rat} (K|\mathbb{Q}) \right) &\subset \mathcal{S}^3(K|\mathbb{Q})_\mathrm{fin}.
\end{align*}
More precisely, for a rational $2$-function $V\in \mathcal{S}_{\mathrm{rat}}^2(K|\mathbb{Q})$ of periodicity $N$ and $\nu \in \frac1\sigma \mathbb{Z}$ and $S= \{p \text{ prim with } p\mid N\}\cup\{2,3\}$,
\begin{align*}
\widetilde V (z) := \frac1\sigma \left( \mathscr{C}_\sigma \left( \Phi^+(\nu, V) \right) \right) \in \mathcal{S}^3(K|\mathbb{Q})_S.
\end{align*}
For $\tilde a^+_n = \left[ \widetilde V(z)\right]_n$, $n\in \mathbb{N}$ we have
\begin{align*}
\mathrm{Frob}_\mathfrak{p} \left( \tilde a_n^+ \right) - \tilde a_{pn}^+ \equiv 0 \mod p^{2 \ord_p(pn)-  \delta_{2,p} + \max\{ 0, \ord_p \left( pn \right) - \gamma_p \}} \mathcal{O}_\mathfrak{p},
\end{align*}
where $\gamma_p$ is equal to $1 + \ord_2(N+1)$, $1$ and $0$ if $p$ is equal to $2$, $3$ and greater than $3$, respectively.
In particular, for unramified $p\geq 5$ in $K|\mathbb{Q}$ with $p\nmid N$, and all $m,r\in \mathbb{N}$,
\begin{align*}
\frob_\mathfrak{p}\left( \tilde a_{mp^{r-1}}^+ \right) - \tilde a_{mp^r}^+ \equiv 0 \mod p^{3r}\mathcal{O}_\mathfrak{p}.
\end{align*}
\end{enumerate}
\end{theorem}

The proofs of these statements go analogously to the non-fractional case.
The Jacobsthal-Kazandzidis congruence then is a special case of \Cref{thm: intro 2-integrality of fractional framing} $(2)$ by taking $a_n= 1$ for all $n\in \mathbb{N}$, that is, taking $V(z) = \frac{z}{1-z}$, and appropriately varying $\sigma$ and $\nu$.

\subsection*{Acknowledgment}
The author is grateful to Johannes Walcher for providing the initial motivation for this work.

\subsection*{Notation}
Throughout this paper, the natural numbers will be meant to be the set of all positive integers, $\mathbb{N}= \{ 1, 2, . . .\}$, while $\mathbb{N}_0 = \mathbb{N}\cup \{0\}$.
If $X$ is a set, then $X^{\mathbb{N}}$ denotes the set of all sequences indexed by the natural numbers, $(x_n)_{n \in \mathbb{N}}\in X^{\mathbb{N}}$.
For a ring $R$ let $R\llbracket z \rrbracket$ denote the ring of formal power series in the variable $z$ with coefficients in $R$.

\section{Preliminaries}\label{Preliminaries}

In this section, we introduce the definitions and notational conventions that will
be used throughout the paper.
We mainly follow the conventions given in \cite{mue20}.
\\

Let $K$ be a fixed algebraic number field and we assume $K$ to be normal over $\mathbb{Q}$.
Denote by $\mathcal{O}$ the ring of integers of $K$.
Let $D$ be the discriminant of $K | \mathbb{Q}$.
We say that a prime $p\in \mathbb{Z}$ is unramified in $K| \mathbb{Q}$ if all prime ideals $\mathfrak{p} \mid p\mathcal{O}$  are unramified.
Note that an unramified prime $p$ is characterized by the property that $p\nmid D$.
For any prime ideal $\mathfrak{p}$, $\mathcal{O}_\mathfrak{p}$ denotes the ring of $\mathfrak{p}$-adic integers.
Then $\mathcal{O}_{\mathfrak{p}}$ is an integral domain and its field of fractions $K_{\mathfrak{p}} = \mathrm{Quot}(\mathcal{O}_{\mathfrak{p}})$ is the $\mathfrak{p}$-adic completion of $K$.

For $\mathfrak{p}\mid (p)$, the \textit{Frobenius element} $\mathrm{Fr}_{\mathfrak{p}}$ at $\mathfrak{p}$ is the unique element satisfying the following two conditions:
$\mathrm{Fr}_{\mathfrak{p}}$ is an element in the decomposition group $D(\mathfrak{p}) \subset \mathrm{Gal}(K/\mathbb{Q})$ of $\mathfrak{p}$ and for all $x \in \mathcal{O}$, $\mathrm{Fr}_{\mathfrak{p}}(x)\equiv x^p \mod \mathfrak{p}$.
By Hensel's Lemma, $\mathrm{Fr}_{\mathfrak{p}}$ can be lifted to $\mathcal{O}_{\mathfrak{p}}$ and then extended to an automorphism $\mathrm{Frob}_\mathfrak{p}\colon K_\mathfrak{p}\rightarrow K_{\mathfrak{p}}$.
By declaring $\mathrm{Frob}_\mathfrak{p}(z)=z$, $\mathrm{Frob}_\mathfrak{p}$ can be (linearly) extended to an endomorphism $\mathrm{Frob}_\mathfrak{p}\colon K_\mathfrak{p}\llbracket z\rrbracket \rightarrow K_\mathfrak{p} \llbracket z\rrbracket$.

In the following, let $R$ be a $\mathbb{Q}$-algebra.
The \textit{Euler operator} $\delta_R \colon R (\hspace{-0.2em}( z )\hspace{-0.2em}) \rightarrow R ( \hspace{-0.2em} ( z ) \hspace{-0.2em} )$ is given by $z\frac{\mathrm{d}}{\mathrm{d}z}$, i.e.
\begin{align*}
	\delta_R \left[ \sum_{n=-\infty}^\infty r_n z^n\right]
	= \sum_{n=-\infty}^\infty n r_n z^n.
\end{align*}
Its (partial) inverse of $\delta_R$ is the \textit{logarithmic integration} $\smallint_R \colon z R \llbracket z \rrbracket \oplus z^{-1} R\left\llbracket z^{-1}\right\rrbracket \rightarrow z R \llbracket z \rrbracket \oplus z^{-1}R\left\llbracket z^{-1}\right\rrbracket$ given by
\begin{align*}
	\smallint\hspace{-0.35em}\,_{R}\left[ \sum_{n=-\infty}^\infty r_n z^n \right]
	= \sum_{n=-\infty}^\infty \frac{r_n} n z^n\qquad\text{and}\qquad \smallint\hspace{-0.35em}\,_R(0)=0.
\end{align*}
For a number $k\in \mathbb{N}$ let $\mathscr{C}_{R,k}$ be the operator $\mathscr{C}_{R,k}\colon R ( \hspace{-0.2em} ( z ) \hspace{-0.2em} ) \rightarrow R ( \hspace{-0.2em} ( z ) \hspace{-0.2em} )$, called the \textit{Cartier operator}, given by
\begin{align*}
	\mathscr{C}_{R,k}\left[\sum_{n=-\infty}^\infty r_nz^n\right]
	=\sum_{n=-\infty}^\infty r_{kn}z^n.
\end{align*}
For a number $\ell \in \mathbb{N}$, let $\varepsilon_{R,\ell} \colon R ( \hspace{-0.2em} ( z ) \hspace{-0.2em} ) \rightarrow R ( \hspace{-0.2em} (z) \hspace{-0.2em} )$ be the $R$-algebra homomorphism uniquely determined by setting
\begin{align*}\varepsilon_{R,\ell} (z) = z^\ell.
\end{align*}
Hereafter, we will omit $R$ from the notation of $\delta_R$, $\smallint_R$, $\mathscr{C}_{R,k}$ and $\varepsilon_{R,\ell}$.

In \cite{walcher16}, an \textit{$s$-function with coefficients in $K$} (for $s \in \mathbb{N}$) is defined to be a formal power series $\widetilde V \in z K\llbracket z \rrbracket$ such that for every unramified prime $p\in \mathbb{Z}$ in $K | \mathbb{Q}$ and prime ideal $\mathfrak{p}\mid (p)$ we have
\begin{align}\label{eq: s-function}
\frac1{p^s}\mathrm{Frob}_\mathfrak{p} \widetilde V \left( z^p \right) - \widetilde V(z) \in z \mathcal{O}_\mathfrak{p} \llbracket z \rrbracket.
\end{align}
In \cite{mue20}, the author introduced \textit{$s$-sequences} as we will repeat immediately.
Such a sequence is basically the coefficients of an $s$-function, listed in a sequence.
This notion generalizes the concept of \textit{$s$-realizable sequences} (a term introduced by Almkvist and Zudilin in \cite{alm06} in 2006) to sequences with algebraic integral coefficients, satisfying certain supercongruences.

A sequence $(a_n)_{n\in \mathbb{N}}\in K^{\mathbb{N}}$ is said to satisfy the \textit{local $s$-function property for $p$}, if $p \in \mathbb{Z}$ is unramified in $K|\mathbb{Q}$, and $a_n \in \mathcal{O}_\mathfrak{p}$ is a $\mathfrak{p}$-adic integer for all $n \in \mathbb{N}$, and
\begin{align}\label{eq: s-sequences}
	\mathrm{Frob}_\mathfrak{p} \left( a_{mp^{r-1}} \right)
	\equiv a_{mp^r} \mod p^{sr}\mathcal{O}_\mathfrak{p},
\end{align}
for all $m,r\in \mathbb{N}$. $(a_n)_{n \in \mathbb{N}}$ is called an \textit{$s$-sequence} if it satisfies the local $s$-function property for all unramified primes $p$ in $K|\mathbb{Q}$.
We denote by $\mathcal{S}^s(K|\mathbb{Q})\subset z \mathcal{O}\left[D^{-1}\right]\llbracket z\rrbracket$ the set of all generating functions of $s$-sequences with coefficients in $K$.
Furthermore, $\overline{\mathcal{S}}^s (K|\mathbb{Q})\subset z K\llbracket z \rrbracket$ denote the set of formal power series, such that $V \in \overline{\mathcal{S}}^s (K|\mathbb{Q})$ if and only if there is a constant $C\in \mathbb{N}$ such that $CV \in \mathcal{S}^s(K|\mathbb{Q})$.
Let $S$ be a finite set consisting of prime numbers, then $\mathcal{S}^s ( K | \mathbb{Q} )_S$ denotes the set of power series such that $V \in \mathcal{S}^s ( K | \mathbb{Q} )_S$ if and only if the coefficients of $V$ satisfy the local $s$-function condition \cref{eq: s-sequences} for all unramified primes $p\not \in S$. 
Also,
\begin{align*}
\mathcal{S}^s(K|\mathbb{Q})_{\mathrm{fin}} = \bigcup_{S} \mathcal{S}^s(K|\mathbb{Q})_S,
\end{align*}
where $S$ runs through all finite subsets of rational primes.
Analogously, the sets $\overline{\mathcal{S}}^s(K|\mathbb{Q})_S$ and $\overline{\mathcal{S}}^s (K|\mathbb{Q})_{\mathrm{fin}}$ are defined.

Let $n \in \mathbb{Z}$ be an integer.
Let $[-]_n$ denote the $R$-functional $[-]_n\colon R ( \hspace{-0.2em} ( z ) \hspace{-0.2em} ) \rightarrow R$, uniquely determined by
\begin{align*}
	\left[z^k\right]_n=\delta_{n,k},
\end{align*}
where $\delta_{n,k}$ denotes the \textit{Kronecker symbol}.
Let $\mathcal{G}$ be the operator $\mathcal{G}\colon \mathbb{C}^\mathbb{N} \rightarrow z\mathbb{C}\llbracket z \rrbracket$, sending the sequence $x \in \mathbb{C}^\mathbb{N}$ to its generating formal power series,
\begin{align*}
\mathcal{G}(x) = \sum_{n=1}^\infty x_n z^n.
\end{align*}
Trivially, the inverse of $\mathcal{G}$ is given by $\left([ - ]_n\right)_{n\in \mathbb{N}}\colon z R\llbracket z \rrbracket \rightarrow R^{\mathbb{N}}$, i.e. $\mathcal{G} \circ \left( [ - ]_n \right)_n = \mathrm{id}_{z R \llbracket z \rrbracket}$ and $\left( [ - ]_k \right)_k \circ \mathcal{G} = \mathrm{id}_{R^{\mathbb{N}}}$.

In other	words, $[-]_n$ extracts the $n$-th coefficient of a Laurent series.
Obviously, for any $V \in R ( \hspace{-0.2em} ( z ) \hspace{-0.2em} )$ and $n \in \mathbb{Z}$, we have
\begin{align}\label{coefofderiv}
	\left[ \delta V(z)\right]_n= n\cdot \left[ V(z) \right]_{n}.
\end{align}
In particular, for $n=0$ we obtain a formula for \textit{integrating by parts}:
Let $F,G\in R ( \hspace{-0.2em} ( z ) \hspace{-0.2em} )$, then
\begin{align*}
	0=\left[\delta(F(z)\cdot G(z))\right]_0
	=\left[G(z)\cdot\delta F(z)+F(z)\cdot \delta
	G(z)\right]_0
\end{align*}
and therefore
\begin{align}\label{intbyparts}
	\left[G(z)\cdot\delta F(z)\right]_0=-\left[ F(z) \cdot \delta G(z)\right]_0.
\end{align}
Analogously, if $[F(z)]_0=0$ and $n \neq 0$, then
\begin{align}\label{coefofintegratedfunction}
\left[\smallint(F(z))\right]_n=\frac1n\left[ F(z) \right]_n
\quad \text{and}\quad
\left[ \smallint F(z) \right]_0 = 0.
\end{align}

\section{Partial Bell Polynomials and Bell Transformations}

Let $\mathbb{Q}[\mathfrak{X}]$ be the ring of polynomials in a countable number of indeterminates $\mathfrak{X}=\{X_1,X_2,...\}$ over $\mathbb{Q}$.
The complete exponential Bell polynomials $\{B_n| n\in \mathbb{N}\}$ (named in honor of the mathematician and science fiction writer Eric Temple Bell) are defined by the generating coefficients of $\exp \left( \sum_{n=1}^\infty \frac{X_i}{i!} z^i \right)$,
\begin{align*}
\exp\left( \sum_{n=1}^\infty \frac{X_i}{i!} z^i\right) =: \sum_{n=1}^\infty B_n(\mathfrak{X}) \frac{z^n}{n!}.
\end{align*}
For $k , n \in \mathbb{N}$, $k \leq n$, the $(n,k)$-th \textit{partial Bell polynomial} $B_{n,k}$ is implicitly defined as the homogeneous part of degree $k$ of the $n$-th complete exponential Bell polynomial $B_n \in \mathbb{Q}[\mathfrak{X}]$.
It is also given by the series expansion
\begin{align*}
	\frac1{k!}\left(\sum_{j=1}^\infty X_j\frac{z^j}{j!}\right)^k
	=\sum_{n=k}^\infty B_{n,k}(\mathfrak{X}) \frac{z^n}{n!}.
\end{align*}
The polynomial $B_{n,k}$ is then explicitly given by
\begin{align*}
	B_{n,k}(\mathfrak{X})=n!\sum_{\alpha\in\pi(n,k)}\left(
	\prod_{i=1}^{n-k+1}
	\frac1{\alpha_i!}\left(\frac{X_i}{i!}\right)^{\alpha_i}\right),
\end{align*}
where $\pi(n,k)$ denotes the set of multi-indices $\alpha\in \mathbb{N}_0^{n-k+1}$ such that
\begin{align*}
	\sum_{i=1}^{n-k+1}\alpha_i=k\quad\text{and}\quad
	\sum_{i=1}^{n-k+1}i\alpha_i=n.
\end{align*}
Note that $B_{n,k}(\mathfrak{X})$ is in fact a polynomial in the variables $X_1 , ... , X_{n-k+1} $ for all $n,k \in \mathbb{N}$, $k \leq n$.
It follows immediately from the definition that the $(n,k)$-th partial Bell polynomial is homogeneous of degree $k$ and of weight $n$, i.e. for a scalar $\lambda\in \mathbb{C}$,
\begin{align*}
B_{n,k}\left(\left(\lambda^i X_i\right)_{i\in \mathbb{N}}\right)
= \lambda^n B_{n,k} ( \mathfrak{X} ).
\end{align*}

For a multi-index $\alpha \in \mathbb{C}^r$ ($r\in \mathbb{N}$), the \textit{absolute value of $\alpha$}  is defined by the sum of components of $\alpha$, i.e. $|\alpha| = \sum_{i=1}^r \alpha_i$.
For a sequence $x = (x_n)_{n\in \mathbb{N}} \in \mathbb{C}^{\mathbb{N}}$ we will use the convention
\begin{align*}
	!x = (n! x_n)_{n\in \mathbb{N}}.
\end{align*}

In \cite{bir18}, \textit{Bell transformations of sequences} were introduced to tackle a wide variety of problems in enumerative combinatorics.
These transformations come along with functional equations satisfied by the corresponding generating power series.
Let $a,b,c,d \in \mathbb{C}$ be fixed.
Then the Bell transformation associated to $(a,b,c,d)$ is a map $\mathscr{Y}_{a,b,c,d}\colon \mathbb{C}^{\mathbb{N}} \rightarrow \mathbb{C}^{\mathbb{N}}$, mapping a sequence $x=(x_n)_{n \in \mathbb{N}} \in \mathbb{C}^{\mathbb{N}}$ to its image $y=(y_n)_{n \in \mathbb{N}}=\mathscr{Y}_{a,b,c,d}(x)$ given by
\begin{align*}
	y_n=\frac1{n!}\sum_{k=1}^n\left[
	\prod_{j=1}^{k-1}(an+bk+cj+d)\right]
	B_{n,k}(!x) \quad\text{for all $n\geq 1$}.
\end{align*}

The main result in \cite{bir18} is the following convolution formula (Thm. 1 in \cite{bir18}):
Let $x,y \in \mathbb{C} ^{\mathbb{N}}$ such that $y = \mathscr{Y}_{a,b,c,d}(x)$.
Assume $c \neq 0$.
Then, for every $n \in \mathbb{N}$ and for any $\lambda \in \mathbb{C}$, we have
\begin{align}\label{eq: birmayer et al convolution}
	\sum_{k=1}^n\left[\prod_{j=1}^{k-1}(\lambda -dj+d)\right]
	B_{n,k}(!y)=\sum_{k=1}^n\left[
	\prod_{j=1}^{k-1}(an+bk+cj+d+\lambda)
	\right]B_{n,k}(!x).
\end{align}
The following statement is a direct consequence of \cref{eq: birmayer et al convolution}.
We give a proof for the sake of completeness.

\begin{theorem}[Consequence of \cref{eq: birmayer et al convolution}]\label{Composition of Bell Transformations}
Let $a,b,c,d,e,f \in \mathbb{C}$ such that either $c \neq 0$ or $b= c = 0$.
Then
\begin{align*}
	\mathscr{Y}_{e, 0, -d, f} \circ \mathscr{Y}_{a, b, c, d} =
	\mathscr{Y}_{a+e, b, c, f}.
\end{align*}
\end{theorem}

\begin{proof}
Let $c\neq 0$ and let $x, y, \hat{y}\in \mathbb{C}^{\mathbb{N}}$ be sequences related by
\begin{align*}
y = \mathscr{Y}_{a,b,c,d}(x)
\quad\text{and}\quad
\hat{y} = \mathscr{Y}_{e,0,-d,f}(y).
\end{align*}
In particular, we have
\begin{align*}
	n!\hat y_n =\sum_{k=1}^n\left[ \prod_{j=1}^{k-1} (en-dj +f) \right] B_{n,k}(!y).
\end{align*}
By \cref{eq: birmayer et al convolution}, we therefore find for $\lambda = en -d +f$
\begin{align*}
\hat{y}_n =
\frac1{n!}\sum_{k=1}^n \left[\prod_{j=1}^{k-1}((a+e)n + bk + cj +f)\right]
B_{n,k}(!x).
\end{align*}
This is the desired formula $\mathscr{Y}_{e, 0, -d, f} \circ \mathscr{Y}_{a, b, c, d} = \mathscr{Y}_{a+e, b, c, f}$.\\
If $b=c=0$, then $\mathscr{Y}_{a, 0, 0, d}^{-1}=\mathscr{Y}_{-a, 0, -d, 0}$ by using \cite[Cor. 2 $(ii)$]{bir18}.
Hence, by the previous case, we may compute
\begin{align*}
\mathscr{Y}_{a+e, 0, 0, f}\circ \mathscr{Y}^{-1}_{a,0, 0, d} =
\mathscr{Y}_{a+e, 0, 0, f} \circ \mathscr{Y}_{-a, 0, -d, 0} =
\mathscr{Y}_{e, 0, -d, f}.
\end{align*}
Equivalently, $\mathscr{Y}_{e,0,-d,f}\circ \mathscr{Y}_{a,0,0,d}=\mathscr{Y}_{a+e,0,0,f}$.
\end{proof}

To us, Bell transformations come in handy to define the framing operators $\Phi^{+/-}$ and use the corresponding functional equations.

\section{The Dwork's Integrality Lemma}

We recall Dwork's Integrality Lemma in the setting of $1$-functions from \cite{walcher16}.
Furthermore, we will prove \Cref{lemma: final countdown} as an extension to \Cref{characterization of 1-functions}, which also contributes to the proof of \Cref{thm: thm 4.34}.

\begin{theorem}[cf. Prop. 7 in \cite{walcher16}, Dwork's Integrality Lemma]\label{characterization of 1-functions}
Let $V \in zK\llbracket z \rrbracket$ and $Y \in 1 + z K\llbracket z \rrbracket$ be related by $V = \log Y$, $Y = \exp (V)$.
Then the following is equivalent
\begin{enumerate}[$(i)$]
\item
$V$ is a $1$-function.
\item
 There is a sequence $q\in \mathcal{O}\left[D^{-1}\right]^\mathbb{N}$ such that
 \begin{align*}
 \smallint V(z) = - \sum_{n=1}^\infty \log (1- q_nz^n)
 \end{align*}
\item
For every unramified prime $p$ in $K|\mathbb{Q}$,
\begin{align*}
\frac{\frob_p(Y)(z^p)}{Y(z)^p} \in 1 + z p \mathcal{O}_\mathfrak{p}\llbracket z \rrbracket,
\end{align*}
\item
$Y \in 1+ z \mathcal{O}\left[ D^{-1} \right] \llbracket z \rrbracket$.
\end{enumerate}
\end{theorem}

\begin{proof}
See  \cite{dwo60} for the classical statement.
In \cite{walcher16} one finds a proof in the setting of $1$-functions.
\end{proof}

Additionally to Dwork's Lemma, we need the following auxiliary in the proof of \Cref{thm: thm 4.34}.

\begin{proposition}\label{lemma: final countdown}
Let $V \in \mathcal{S}^1(K|\mathbb{Q})$ and $Y\in 1+ zK \llbracket z \rrbracket$ be related by $V=\delta \log Y$, $Y=\exp( \smallint V)$ and let $p$ be unramified in $K|\mathbb{Q}$ and let $\mathfrak{p}\subset \mathcal{O}$ be a prime ideal dividing $(p)$.
Then
\begin{align}\label{eq: 3-framing integrality}
\left[Y(z)^n\right]_m \equiv 0 \mod p^{\max\{ 0, \ord_p(n)-\ord_p(m) \}} \mathcal{O}_\mathfrak{p}.
\end{align}
\end{proposition}

\begin{proof}
Write
\begin{align*}
\exp(\smallint V(z))= 1+ \sum_{m=1}^\infty y_m z^m
\quad\text{and}\quad
\exp(n\smallint V(z)) = 1 + \sum_{m=1}^{\infty} \tilde y_m z^m.
\end{align*}
Of course, by Dwork's Integrality \Cref{characterization of 1-functions} $\tilde y _m$ and $ y_m$ are elements $\mathcal{O}_\mathfrak{p}$ for all $m\in \mathbb{N}$.
In particular, by using the functional equations satisfied by Bell transformations given in \cite{bir18}, we obtain $\left( \frac{\tilde y_m}n \right)_{m\in \mathbb{N}} = \mathscr{Y}_{0,0,-1,n}\left((y_m)_{m\in \mathbb{N}}\right) $.
Explicitely,
\begin{align}\label{eq: 200724}
\tilde y_m = n \sum_{k=1}^m \frac1k\binom{n-1}{k-1}\frac{k!}{m!} B_{m,k}(!y).
\end{align}
Note, that $\displaystyle \binom{n-1}{k-1}\in \mathbb{N}_0 \text{ and } \frac{k!}{m!} B_{m,k}(!y) \in \mathcal{O}_\mathfrak{p}$, since $y \in \mathcal{O}_\mathfrak{p}^\mathbb{N}$.
Therefore, we have for all $1 \leq k\leq m$ with $\ord_p(k)\leq \ord_p(m)$
\begin{align*}
\frac nk \binom{n-1}{k-1} \frac{k!}{m!} B_{m,k}(!y) &\equiv 0 \mod p^{\ord_p(n)-\ord_p(k)} \mathcal{O}_\mathfrak{p} \\
&\equiv 0 \mod p^{\ord_p(n)-\ord_p(m)}\mathcal{O}_\mathfrak{p}.
\end{align*}
Hence, mod $p^{\ord_p(n)-\ord_p(m)}\mathcal{O}_\mathfrak{p}$, we can ignore those sumands in \cref{eq: 200724} where $\ord_p(k)\leq \ord_p(m)$.
Let $1\leq k\leq m$ with $\ord_p(k)>\ord_p(m)$.
In that case, we will show that
\begin{align}\label{20072422}
\frac{k!}{m!} B_{m,k}(!y) \equiv 0 \mod p^{\ord_p(k)- \ord_p(m)}\mathcal{O}_\mathfrak{p},
\end{align}
which implies \cref{eq: 3-framing integrality}.
We have
\begin{align}\label{eq: explicit partial bell polynomials}
\frac{k!}{m!} B_{m,k}(!y) = \sum_{\alpha \in \pi (m,k)} \binom{k}{\alpha_1, \dots , \alpha_{m-k+1}} \prod_{i=1}^{m-k+1} y_i^{\alpha_i},
\end{align}
where $\pi(m,k)\subset \mathbb{N}_0^{m-k+1}$ such that $\alpha \in \pi(m,k)$ if and only if
\begin{align*}
\sum_{i=1}^{m-k+1}\alpha_i = k
\quad \text{and}\quad
\sum_{i=1}^{m-k+1} i\alpha_i=m.
\end{align*}
Let $\alpha\in \pi(m,k)$.
Assume there is an $1\leq j\leq m-k+1$ such that $\ord_p(\alpha_j)\leq \ord_p(m)$. 
Then
\begin{align*}
\binom{k}{\alpha_1, \dots , \alpha_{m-k+1}}
&= \frac{k}{\alpha_j} \binom{k-1}{\alpha_1, \dots , \alpha_j-1, \dots , \alpha_{m-k+1}}\\
&\equiv 0 \mod p^{\ord_p(k)-\ord_p(\alpha_j)}\mathcal{O}_\mathfrak{p}\\
&\equiv 0 \mod p^{\ord_p(k)-\ord_p(m)}\mathcal{O}_\mathfrak{p}.
\end{align*}
Hence, mod $p^{\ord_p(n)-\ord_p(k)} \mathcal{O}_\mathfrak{p}$, we can ignore these sumands in \cref{eq: explicit partial bell polynomials}.
Suppose, that there exists an $\alpha \in \pi(m,k)$ such that for all $1 \leq i \leq m-k+1$ we have $\ord_p(\alpha_i) > \ord_p(m)$.
Then
\begin{align*}
\ord_p(m)= \ord_p\left(\sum_{i=1}^{m-k+1} i \alpha_i\right) \geq \min_{i=1,..., m-k+1} \ord_p( i \alpha_i ) > \ord_p(m),
\end{align*}
which is a contradiction.
We conclude
\begin{align*}
\tilde y_m \equiv 0 \mod p^{\max\{ 0 , \ord_p(n) - \ord_p(m) \}}\mathcal{O}_\mathfrak{p}
\end{align*}
in every case.
\end{proof}

\section{Bell Transformations and Framing}

In this section we will define the framing operators $\Phi^{+/-}$ as operators on formal power series via Bell transformations.
\Cref{prop: framing operators and properties} ensures that the framing operators fulfill the functional equation \cref{eq: intro framing op phi-}, which is a consequence of the theory of Bell transformations, and gives a formula for the coefficients.

\begin{definition}[framing operators $\Phi^{+/-}$]\label{framing}
Define the framing operator $\Phi^+ \colon \mathbb{C} \times z\mathbb{C} \llbracket z \rrbracket \rightarrow \mathbb{C}\llbracket z \rrbracket$, $(\nu,V ) \mapsto V^{(+,\nu)}(z)$ by the following composition
\begin{align*}
\Phi^+(\nu, - )\colon z \mathbb{C}\llbracket z\rrbracket
\xrightarrow{\smallint}
z \mathbb{C}\llbracket z\rrbracket
\xrightarrow{([-]_n)_{n\in\mathbb{N}}}
\mathbb{C}^{\mathbb{N}}
\xrightarrow{\mathscr{Y}_{\nu, 0,0,0}}
\mathbb{C}^{\mathbb{N}}
\xrightarrow{\mathcal{G}}
z \mathbb{C}\llbracket z\rrbracket
\xrightarrow{\delta}
z \mathbb{C}\llbracket z\rrbracket.
\end{align*}
Also, define $\Phi^-\colon \mathbb{C}\times z \mathbb{C}\llbracket z \rrbracket \rightarrow z \mathbb{C} \llbracket z \rrbracket$, $(\nu, V) \mapsto \Phi^-(\nu, V)$ by twisting sign convolution $z\mapsto (-1)^\nu z$, i.e.
\begin{align*}
	\Phi^-(\nu, V) = V^{(+,\nu)} \left( (-1)^\nu z \right).
\end{align*}
\end{definition}

\begin{proposition}\label{prop: framing operators and properties}
Let $V \in z \mathbb{C}\llbracket z \rrbracket$ and write $V^{(+,\nu)}:= \Phi^+(\nu, V)$ and $V^{(-,\nu)}:=\Phi^-(\nu, V)$.
Furthermore, write $a^+_n := \left[ V^{(+,\nu)} (z) \right]_n$ and $a^-_n := \left[ V^{(-,\nu)} (z) \right]_n$.
Then

\begin{enumerate}[$(i)$]
\item $\Phi^+$ and $\Phi^-$ define group actions of the additive group $(\mathbb{C}, +)$ on the set $z \mathbb{C}\llbracket z\rrbracket$ of formal power series with vanishing zeroth coefficient.
In particular, we have
\begin{align*}
\Phi^{+/-}(0,-) = \mathrm{id} \quad \text{and} \quad \Phi^{+/-}(\nu,-) \circ \Phi^+(\mu, -) = \Phi^{+/-} (\nu + \mu, - ).
\end{align*}

\item
The following functional equations are satisfied
\begin{align}\label{eq: framing op phi+}
\smallint V^{(+,\nu)}\left( z\exp(-\nu \smallint V(z)) \right) = \smallint V(z),
\end{align}
and
\begin{align} \label{eq: framing op phi-}
\smallint V^{(-,\nu)} \left( z\left( - \exp(- \smallint V(z))^\nu \right) \right) = \smallint V(z).
\end{align}

\item
For the coefficients $a_n^+$ and $a_n^-$ we have for all $n\in \mathbb{N}$,
\begin{align}\label{eq: framed coefficients}
a_n^+ = \frac1\nu\left[\frac{\exp(\nu n \smallint V(z))}{z^n}\right]_0,
\end{align}
and consequently by definition,
\begin{align}\label{eq: framed coefficients -}
a_n^-
= (-1)^{\nu n}a_n^+
=\frac{(-1)^{\nu n}}\nu \left[\frac{\exp(\nu n \smallint V(z))}{z^n}\right]_0.
\end{align}

\end{enumerate}
\end{proposition}

\begin{proof}
The group action property for $\Phi^+$ follows immediately from \Cref{Composition of Bell Transformations} by setting $a=\nu$, $e=\mu$ and $b=c=d=f=0$.
Since the partial Bell polynomials $B_{n,k}(\mathfrak{X})$, $k\leq n$, have weight $n$ , it is obvious that the additional sign change does not effect the group action property of $\Phi^+$, i.e. $\Phi^+$ passes its group action property on to $\Phi^-$.
This proves $(i)$.

The functional equation \cref{eq: framing op phi+} is given by \cite[Cor. 4 $(iv)$]{bir18}.
By using the \textit{Lagrange Inversion Formula} (LIF) given below, we find the   formulas given in \cref{eq: framed coefficients} and \cref{eq: framed coefficients -}.
For further reference of the LIF, see for instance \cite{ges16}, \cite{mer06}.

\begin{theorem}[LIF]\label{LIF}
Let $F,H \in z \mathbb{C}\llbracket z \rrbracket$ and $G \in z \mathbb{C} \llbracket z \rrbracket$ the compositional inverse to $F$, i.e. $F(G(z)) = G(F(z)) = z$.
Then
\begin{align}\label{glif}
	[H(G(z))]_n=\frac1n
	\left[\frac{\delta H(z)}{F(z)^n}\right]_0.
\end{align}
\end{theorem}

Of course, \cref{eq: framed coefficients -} follows from \cref{eq: framed coefficients} by definition.
Therefore, it is sufficient to proof \cref{eq: framed coefficients}.
For
\begin{align*}
F(z)= z\exp (-\nu \smallint V(z))
\end{align*}
let $G \in z \mathbb{C} \llbracket z\rrbracket$ be the compositional inverse to $F$, $F(G(z)) = G(F(z)) = z$.
Hence,
\begin{align*}
	\smallint V^{(+,\nu)}\left( z \right) = \smallint V( G(z) )
\end{align*}
Using \cref{coefofintegratedfunction}, we have
\begin{align*}
	a^+_n = \left[ V^{(+,\nu)}(z)\right]_n
	= n \left[\smallint V^{(+,\nu)}(z)\right]_n.
\end{align*}
Then, \Cref{LIF} gives
\begin{align*}
a^+_n = \left[\frac{V(z)}{z^n}\exp(\nu n \smallint V(z))\right]_0.
\end{align*}
Since $[-]_0\circ \delta \equiv 0$ (compare with \cref{coefofderiv}) we obtain
\begin{align}\label{eq: partial integration for framing coefficients}
	0 = \left[\delta \left( \frac{\exp(\nu n \smallint V(z))}{z^n}\right) \right]_0
	= n\cdot \left[ \frac{\nu V(z) - 1}{z^n}\exp(\nu n \smallint V(z)) \right]_0.
\end{align}
Therefore,
\begin{align*}
	a_n^+ = \left[\frac{V(z)}{z^n}\exp(\nu n \smallint V(z))\right]_0
	= \frac1\nu\left[\frac{\exp(\nu n \smallint V(z))}{z^n}\right]_0,
\end{align*}
proving $(iii)$.

Let $\widetilde V(z) \in z \mathbb{C}\llbracket z \rrbracket$  be the power series satisfying the functional equation \cref{eq: framing op phi-}, i.e.
\begin{align*}
\smallint \widetilde V \left( z\left( - \exp(- \smallint V(z))^\nu \right) \right) = \smallint V(z),
\end{align*}
and write $\tilde a_n := \left[ \widetilde V(z) \right]_n$ for all $n\in \mathbb{N}$.
Then, by an analogue calculation as for $a^+_n$ we find
\begin{align*}
\tilde a_n&=\frac{(-1)^{\nu n}}\nu \left[\frac{\exp(\nu n \smallint V(z))}{z^n}\right]_0
= a_n^-, \quad \text{for all $n\in \mathbb{N}$}.
\end{align*}
Hence, $\widetilde V = V^{(-,\nu)}$, proving $(ii)$.
\end{proof}

\section{Wolstenholme's Theorem: Harmonic Sums, Binomials and a new Generalization}\label{section: wolstenholme's theorem}

The goal of the present section is to prove a generalization of Wolstenholme's Theorem given by \Cref{lemma: general wolstenholme}.
For a survey on Wolstenholme's Theorem see \cite{mes11}.

In 1862, J. Wolstenholme proved that for all primes $p \geq 5$ we have
\begin{align}\label{eq: wolstenholme's original thm}
	\binom{2p-1}{p-1} \equiv 1 \mod p^3.
\end{align}
This result is originally known as Wolstenholme's theorem, see \cite{wol1862} for the original work.
As pointed out by Rosen in \cite{ros13}, the related congruence on harmonic numbers $H_n:= \sum_{k=1}^n\frac1k$, stating that for all primes $p\geq 5$,
\begin{align}\label{eq: wolstenholme's thm}
	H_{p-1} = \sum_{k=1}^{p-1} \frac1k \equiv 0 \mod p^2
\end{align}
(which was discovered 80 years earlier by E. Waring in 1782 (see \cite{war1782}) and later by C. Babbage in 1819 (see \cite{bab1819})), is in fact equivalent to Wolstenholme's original result.
In modern literature, \cref{eq: wolstenholme's thm} is referred to as \textit{Wolstenholme's Theorem}.
More generally, we have
\begin{theorem}[“Wolstenholme's Theorem”, Waring-Babbage, see for instance \cite{gessel83}]\label{thm: waring-babbage}
Let $p$ be a prime and let $\epsilon_p$ be $2$, $1$, or $0$ according to whether $p$ is $2$, $3$ or $\geq 5$, respectively. Then, for all $n\in \mathbb{N}$,
\begin{align}\label{eq: waring-babbage}
\sum_{\substack{k=1\\ p\nmid k}}^n \frac1k \equiv 0 \mod p^{\max \{ 0 , 2 \ord_p(n) - \epsilon_p \}}\mathbb{Z}_p.
\end{align}
\end{theorem}
\begin{proof}
First check the identity, for $1\leq k\leq n$,
\begin{align*}
\frac1k +\frac1{n-k} = -\frac n{k^2} + \frac{n^2}{k^2(n-k)}.
\end{align*}
Note that the sum given in \cref{eq: waring-babbage} is trivially a $p$-adic integer.
Therefore, w. l. o. g., we assume $2\ord_p(n)-\epsilon_p \geq 0$.
Then,
\begin{align*}
2 \sum_{\substack{k=1\\ p\nmid k}}^n \frac1k
&= \sum_{\substack{k=1 \\ p\nmid k}}^n \left( \frac1k + \frac1{n-k}\right)
= \sum_{\substack{k=1 \\ p\nmid k}}^n \left(-\frac n{k^2} + \frac{n^2}{k^2(n-k)}\right)\\
&= -n \sum_{\substack{k=1 \\ p\nmid k}}^n \frac1{k^2} + n^2
\sum_{\substack{k=1 \\ p\nmid k}}^n \frac1{k^2(n-k)}
\equiv -n \sum_{\substack{k=1 \\ p\nmid k}}^n \frac1{k^2} \mod p^{2\ord_p(n)} \mathbb{Z}_p.
\end{align*}
Now, we immediately observe that the assertion \cref{eq: waring-babbage} is equivalent to the validity of the following congruence,
\begin{align}\label{eq: wolstenholme's dilogarithmic sum}
\sum_{\substack{k=1 \\ p\nmid k}}^n \frac1{k^2} \equiv 0 \mod p^{\ord_p(n)-\epsilon_p +\delta_{p,2}} \mathbb{Z}_p.
\end{align}
A proof of \cref{eq: wolstenholme's dilogarithmic sum} is given in \cite[Lemma 1]{gessel83}.
What is more, we will prove  \Cref{lemma: general wolstenholme}, which is a generalization of \cref{eq: wolstenholme's dilogarithmic sum} involving algebraic coefficients related to (rational) $2$-functions.
In particular, \cref{eq: wolstenholme's dilogarithmic sum} follows from \Cref{lemma: general wolstenholme} for $V(z)=\frac z{1-z}$ for $p\geq 3$ and from \Cref{remark: wolstenholme for p=2} for $p=2$.
\end{proof}

There are a number of generalizations and extensions of Wolstenholme's Theorem in terms of multiple harmonic sums and congruences among binomial coefficients.
The next theorem gives a generalization in yet another direction.
We will allow the nominator each summand be the folding of an periodic sequence with algebraic coefficients.
The motivation for this has its origin in the proof of \Cref{thm: thm 4.34}.

\begin{theorem}\label{lemma: general wolstenholme}
Let $p$ be an unramified prime in $K|\mathbb{Q}$.
Let $(a_k)_{k\in \mathbb{N}} \in \mathcal{O}_\mathfrak{p}^\mathbb{N}$  be a periodic sequence of periodicity $N$, i.e. $N\in \mathbb{N}$ is given by
\begin{align*}
N= \min \{ i\in \mathbb{N} \,|\, a_{k + i} = a_k \text{ for all $n\in \mathbb{N}$} \}.
\end{align*}
Then, for all $n \in \mathbb{N}$,
\begin{align*}
\sum_{\substack{k =1 \\ p\nmid k}}^n \frac{a_{n-k}a_k}{k^2}\equiv 0 \mod p^{ \max \{ 0 , \ord_p(n) - \epsilon_{p,N} \}}\mathcal{O}_\mathfrak{p},
\end{align*}
where
\begin{align*}
\epsilon_{p,N}= 
\begin{cases}
\max \{ \ord_2(N) , \ord_2(N+2) \} ,& \text{if $p=2$ and $2 \mid N$},\\
1+\ord_2(N+1),& \text{if $p=2$ and $2\nmid N$},\\
1+ \ord_3(N),&\text{if $p=3$,}\\
\ord_p(N),& \text{if $p\geq 5$}.
\end{cases}
\end{align*}
\end{theorem}

\begin{proof}
Write $n=mp^r$ for $r=\ord_p(n)$ and suitable $m\in \mathbb{N}$ such that $\gcd(m,p)=1$.
Then, by using the geometric series $(1- xp)^{-1} = \sum_{k=0}^\infty (xp)^k$ for $x\in \mathbb{Z}_p$, we obtain
\begin{align}\label{eq: general wolstenholme sum permutation}
\sum_{\substack{k=0\\ p\nmid k}}^n\frac{a_{n-k}a_k}{k^2}
&\overset{k\mapsto \mu p^r + \ell}{=} \sum_{\mu=0}^{m-1}\sum_{\substack{\ell =0 \\ p\nmid \ell}}^{p^r} \frac{a_{(m-\mu)p^r - \ell}a_{\mu p^r + \ell}}{(\mu p^r + \ell)^2}
= \sum_{\mu=0}^{m-1}\sum_{\substack{\ell =0 \\ p\nmid \ell}}^{p^r} \frac{a_{(m-\mu)p^r - \ell}a_{\mu p^r + \ell}}{\ell^2(1 +\frac\mu\ell p^r)^2}\nonumber\\
&\hspace{1.25em}\equiv \sum_{\mu=0}^{m-1}\sum_{\substack{\ell=0 \\ p\nmid \ell}}^{p^r} \frac{a_{(m-\mu)p^r -\ell} a_{\mu p^r + \ell}}{\ell^2} \mod p^r \mathcal{O}_\mathfrak{p}.
\end{align}
Note, that the sum in \cref{eq: general wolstenholme sum permutation} is trivially an element in $\mathcal{O}_p$.
First, find $q \in \mathbb{N}$ such that $p \nmid q$, $N \mid q-1$, and -- whenever possible -- $p \nmid q+1$.
We have
\begin{itemize}
\item[\textit{Case 1:}]
If $p\nmid N+1$ and $p\nmid N+2$, choose $q = N+1$.
Trivially, $N\mid q - 1$.
In that case, $p \nmid q$, by definition, and
\begin{align*}
q^2 - 1 = (q-1)(q+1) = N (N+2),
\end{align*}
and therefore,
\begin{align*}
\ord_p(q^2-1) = \ord_p(N).
\end{align*}

\item[\textit{Case 2:}]
Let $p> 2$. If $p\mid N+1$ and $p \nmid N+2$, then choose $q=3N+1$.
Note that $p\mid N+1$ implies $p \nmid N$.
Indeed, $N \mid q-1$ and
\begin{align*}
q &= 3N+1 \equiv 2N \not\equiv 0 \mod p, \quad \text{and}\\
q+1 &= 3N + 2 \equiv N \not\equiv 0 \mod p, \quad \text{since $p\nmid N$ and  $p\neq 2$}.
\end{align*}
Also, $p \nmid q$, since $3N+1 = N-1 + 2(N+1)$ and $p\neq 2$.
Finally, $p\nmid q+1$, since $3N+2 = N + 2(N+1)$ and $p\nmid N$.
In this case,
\begin{align*}
q^2-1 = 3N (3N + 2 ).
\end{align*}
Hence,
\begin{align*}
\ord_p( q^2 - 1 ) = \ord_p(N) + \delta_{p,3}.
\end{align*}

\item[\textit{Case 3:}]
Let $p=2$ and $p\mid N+1$, $p\nmid N+2$.
Then choose $q = 2N+1$.
Observe, that $N \mid q-1$ and
\begin{align*}
q &= 2N + 1 \equiv N \neq 0 \mod 2, \quad \text{since $p\nmid N$}.
\end{align*}
At the same time,
\begin{align*}
\ord_2(q^2 - 1) = \ord_2((q-1)(q+1))= \ord_2(4N (N+1)) = 2 + \ord_2(N+1).
\end{align*}

\item[\textit{Case 4:}]
Let $p \not\in\{2,3\}$ and $p \nmid N+1$ and $p \mid N +2$, then choose $q = 2 N + 1$.
Trivially, $N| q - 1$ and we have
\begin{align*}
q &= 2 N + 1 \equiv -3 \not\equiv 0 \mod p, \quad\text{since $p\neq 3$ and,}\\
q+1 &= 2 N + 2 \equiv N \not \equiv 0 \mod p, \quad \text{since $p \neq 2$}.
\end{align*}
In that case,
\begin{align*}
q^2 - 1 = 4N (N+ 1).
\end{align*}
Hence,
\begin{align*}
\ord_p(q^2-1) = \ord_p(N) .
\end{align*}

\item[\textit{Case 5:}]
Let $p= 3$, $3\nmid N+1$ and $3\mid N+2$, then choose $q= 3N + 1$.
Note that $ p=3 $ implies $3\nmid N$.
Hence, $N \mid q-1$ and we have
\begin{align*}
q &= 3N+1\equiv N -3 \equiv N \not \equiv 0 \mod 3.
\end{align*}
Furthermore,
\begin{align*}
q^2 - 1 =3N(3N+2)
\end{align*}
and therefore,
\begin{align*}
\ord_3(q^2-1) = 1.
\end{align*}

\item[\textit{Case 6:}]
Let $p = 2$, $2 \nmid N+1$ and $2 \mid N + 2$ (i.e. $2\mid N$), then choose $q= N + 1$.
We have
\begin{align*}
	q= N+1 \not \equiv 0 \mod 2.
\end{align*}
Then
\begin{align*}
q^2 -1 = N^2 +2N = N(N+2)
\end{align*}
and therefore
\begin{align*}
\ord_2(q^2-1) = \ord_2(N)+\ord_2(N+2)= 1 + \max\{ \ord_2(N),\ord_2(N+2)\}.
\end{align*}

\end{itemize}
Since we may find $q$ such that $q \equiv 1 \mod N$ in every case, we have $a_{m+ q\ell}= a_{m+\ell}$ for all $m\in \mathbb{N}_0$ and $\ell\in \mathbb{N}$.
Since $p\nmid q$, we see that multiplication  by $q$ mod $p^{r}$ gives a bijection on $\left(\mathbb{Z}/p^r\mathbb{Z}\right)^\times$ and hence, we may also permute the sumands in \cref{eq: general wolstenholme sum permutation} by the transformation $\ell \mapsto q\ell$.
Therefore,
\begin{align*}
\sum_{\substack{k=0\\ p\nmid k}}^n\frac{a_{n-k}a_k}{k^2}
&\equiv\sum_{\mu=0}^{m-1}\sum_{\substack{\ell=0 \\ p\nmid \ell}}^{p^r} \frac{a_{(m-\mu)p^r -\ell} a_{\mu p^r + \ell}}{\ell^2} \mod p^r \mathcal{O}_\mathfrak{p}\\
&\equiv\sum_{\mu=0}^{m-1}\sum_{\substack{\ell=0 \\ p\nmid \ell}}^{p^r} \frac{a_{(m-\mu)p^r -q\ell} a_{\mu p^r + q\ell}}{(q\ell)^2}\mod p^r \mathcal{O}_\mathfrak{p}\\
&= \frac1{q^2}\sum_{\mu=0}^{m-1}\sum_{\substack{\ell=0 \\ p\nmid \ell}}^{p^r} \frac{a_{(m-\mu)p^r -\ell} a_{\mu p^r + \ell}}{\ell^2}\equiv
\frac1{q^2}\sum_{\substack{k=0\\ p\nmid k}}^n\frac{a_{n-k}a_k}{k^2} \mod p^r\mathcal{O}_\mathfrak{p}.
\end{align*}
Equivalently,
\begin{align*}
\frac{q^2-1}{q^2} \cdot \sum_{\substack{k=0 \\ p \nmid k}}^n \frac{a_k a_{n-k}}{k^2}
\equiv 0 \mod p^r \mathcal{O}_\mathfrak{p}.
\end{align*}
By the above choice of $q$ and recalling $q^2-1 \equiv 0 \mod p^{\epsilon_{p,N} + \delta_{p,2}} \mathbb{Z}$, we therefore conclude
\begin{align*}
\sum_{\substack{k=0 \\ p \nmid k}}^n \frac{a_k a_{n-k}}{k^2}
\equiv 0 \mod p^{r-\epsilon_{p,N}-\delta_{2,p}}\mathcal{O}_\mathfrak{p}.
\end{align*}
For $p>2$, we are finished.
For $p=2$ we may in particular assume $\ord_2(n)=r\geq 1$.
By using the symmetry (i.e. the invariance of $k\mapsto n-k$) of the coefficients $a_ka_{n-k}$, we have
\begin{align*}
\sum_{\substack{k=0 \\ k\text{ odd}}}^n \frac{a_{n-k}a_k}{k^2}\equiv 2 \cdot \sum_{ \substack{ k=0 \\ k\text{ odd}}}^{\nicefrac n2} \frac{a_{n-k} a_k}{k^2} \mod 2^r \mathcal{O}_2.
\end{align*}
Then by the same calculation as for general $p$, and the same choice of $q\in \mathbb{Z}$, we find
\begin{align*}
\sum_{ \substack{ k=0 \\ k\text{ odd}}}^{\nicefrac n2} \frac{a_{n-k} a_k}{k^2} \equiv \frac1{q^2} \sum_{ \substack{ k=0 \\ k\text{ odd}}}^{\nicefrac n2} \frac{a_{n-k} a_k}{k^2}\mod 2^r \mathcal{O}_2.
\end{align*}
Equivalently,
\begin{align*}
\frac{q^2 - 1}{q^2} \sum_{ \substack{ k=0 \\ k\text{ odd}}}^{\nicefrac n2} \frac{a_{n-k} a_k}{k^2} \equiv 0 \mod 2^r\mathcal{O}_2.
\end{align*}
Therefore,
\begin{align*}
\sum_{\substack{k=0 \\ k\text{ odd}}}^n \frac{a_{n-k}a_k}{k^2} \equiv 0 \mod 2^{r - \epsilon_{ p , N } } \mathcal{O}_2,
\end{align*}
as stated.
\end{proof}

\begin{remark}[$p=2$]\label{remark: wolstenholme for p=2}
In the special case of \cref{eq: wolstenholme's dilogarithmic sum}, for $p=2$ and $V(z)=\frac z{1-z}$ (i.e. $a_n=1$ for all $n\in \mathbb{N}$) one can improve the $2$-adic estimation.
In that case, we find
\begin{align}\label{eq: p=2 classic wolstenholme}
\sum_{\substack{k=1 \\ k \text{ odd}}}^n \frac1{k^2} \equiv 0 \mod 2^{\ord_2(n)-1} \mathbb{Z}_2,
\end{align}
which is sharper than what \Cref{lemma: general wolstenholme} permits.
The reason for this is given by \cref{eq: remark wolstenholme p=2} below.
We prove \cref{eq: p=2 classic wolstenholme} for the sake of completeness.
Write $n=2^rm$ for $r=\ord_2(n)$ and $m\in\mathbb{N}$, $\gcd(2,m)=1$.
Since
\begin{align*}
\sum_{\substack{k=1 \\ k \text{ odd}}}^n \frac1{k^2}
= \sum_{\mu=0}^{m-1}\sum_{\substack{\ell=0 \\ \ell\text{ odd}}}^{2^r}\frac1{(\mu \cdot 2^r +\ell)^2}
\equiv \sum_{\mu=0}^{m-1}\sum_{\substack{\ell=0 \\ \ell\text{ odd}}}^{2^r}\frac1{\ell^2} 
= m \cdot \sum_{\substack{k=0 \\ k \text{ odd}}}^{2^r} \frac1{k^2} \mod 2^r \mathbb{Z}_2,
\end{align*}
we may assume w.l.o.g. $n=2^r$.
For $r=1$ and $r=2$ the assertion is trivial.
Therefore, we may also assume $r\geq 3$.
In that case, every odd square $k^2$ has four square roots modulo $2^r$, namely, $\pm k$ and $2^{r-1} \pm k$.
Therefore,
\begin{align}\label{eq: remark wolstenholme p=2}
\sum_{\substack{k=0 \\ k \text{ odd}}}^{2^r} \frac1{k^2}
\equiv 4\cdot \sum_{\substack{k=0 \\ k \text{ odd}}}^{2^{r-2}} \frac1{k^2}\mod 2^r.
\end{align}
Furthermore, the multiplication $k\mapsto 3k$ gives a bijection on $ \left( \mathbb{Z} / 2^r \mathbb{Z} \right) ^\times$ and
\begin{align*}
\sum_{\substack{ k=0 \\ k \text{ odd}}}^{2^{r-2}} \frac1{k^2}
\equiv \sum_{\substack{ k=0 \\ k \text{ odd}}}^{2^{r-2}} \frac1{(3k)^2} 
= \frac19 \cdot \sum_{\substack{ k=0 \\ k \text{ odd}}}^{2^{r-2}} \frac1{k^2}\mod 2^r.
\end{align*}
Equivalently,
\begin{align*}
\frac89\cdot \sum_{\substack{ k=0 \\ k \text{ odd}}}^{2^{r-2}} \frac1{k^2}
\equiv 0 \mod 2^r.
\end{align*}
Hence,
\begin{align}\label{eq: remark wolstenholme p=2 - 2}
\sum_{\substack{k=0 \\ k\text{ odd}}}^{2^{r-2}} \frac1{k^2}\equiv 0 \mod 2^{r-3}.
\end{align}
Inserting \cref{eq: remark wolstenholme p=2 - 2} in \cref{eq: remark wolstenholme p=2} leads to \cref{eq: p=2 classic wolstenholme}.
\end{remark}

We will now state the so-called \textit{Jacobsthal-Kazandzidis congruence} (\Cref{thm: jacobsthal-kazandzidis}) which was first discovered by Jacobsthal as a corollary to his work \cite{bru52} in 1949 and later in a more general formulation by Kazandzidis in 1969 (see \cite{kaz69}) and Trakhtman in 1974 (see \cite{tra74}).
Nonetheless, the proof of \Cref{thm: jacobsthal-kazandzidis} as given in \cite{gessel83} makes use of the congruence relations of harmonic sums as stated by \Cref{thm: waring-babbage}, \Cref{lemma: general wolstenholme} and \Cref{remark: wolstenholme for p=2}.
The Jacobsthal-Kazandzidis congruence also follows from \Cref{thm: intro 2-integrality of fractional framing} (2) as we will see in \Cref{section: fractional framing and examples}.
Moreover, the proof of \Cref{thm: thm 4.34} may be considered as a generalization of the proof of \Cref{thm: jacobsthal-kazandzidis}.
Therefore, the Jacobsthal-Kazandzidis congruence can be considered to be a prototype of the statements \Cref{thm: thm 4.34} and \Cref{thm: intro 2-integrality of fractional framing} (2).

\begin{theorem}[Jacobsthal-Kazandzidis]\label{thm: jacobsthal-kazandzidis}
Let $a,b\in \mathbb{N}_0$ be non-negative integers, $r\in \mathbb{N}$ a positive integer, and let $p$ be a prime.
Then we have
\begin{align*}
\binom{ap^r}{bp^r} \equiv \binom{ap^{r-1}}{bp^{r-1}} \mod p^{3 r - \epsilon_p},
\end{align*}
where $\epsilon_p$ is (as in \Cref{thm: waring-babbage}) $2$, $1$, or $0$, whether $p$ is $2$, $3$, or greater than $3$, respectively.
\end{theorem}

\begin{proof}
We begin with
\begin{align}\label{eq: beware! this can be generalized!}
\binom{a p^r}{b p^r}\hspace{-0.25em}\left/\binom{a p^{r-1}}{b p^{r-1}}\right.
&= \prod_{k=1}^{bp^r}\frac{(a-b)p^r +k}k\cdot \prod_{k=1}^{bp^{r-1}} \frac k{(a-b)p^{r-1} + k}\nonumber\\
&= \prod_{\substack{k=1 \\ p\nmid k}}^{bp^r} \left(1 + p^r \frac{a-b}k\right)\nonumber\\
&\equiv 1 + p^r(a-b)F_1 + p^{2r}(a-b)^2 F_2 \mod p^{3r},
\end{align}
where $F_1$ and $F_2$ are given by the harmonic sums
\begin{align*}
F_1= \sum_{ \substack{ k=0 \\ p\nmid p} }^{bp^r} \frac1k \quad \text{and} \quad F_2 = \sum_{\substack{ i,j=0, i<j \\ p\nmid ij }}^{bp^r} \frac1{ij}.
\end{align*}
We have
\begin{align*}
2F_2 = \sum_{\substack{i\neq j \\ p\nmid ij}}^{bp^r} \frac1{ij}
= \left[ \sum_{i=1 ,\, p\nmid i}^{bp^r}\frac1i \right]^2 - \sum_{\substack{i=1 \\p\nmid i}}^{bp^r} \frac1{i^2}.
\end{align*}
By \Cref{thm: waring-babbage}, \Cref{lemma: general wolstenholme} and \Cref{remark: wolstenholme for p=2}, this implies
\begin{align*}
F_2\equiv 0 \mod p^{r-\varepsilon_p},
\end{align*}
and finally,
\begin{align*}
\binom{ap^r}{bp^r}\hspace{-0.25em}\left/\binom{ap^{r-1}}{bp^{r-1}}\right. \equiv 1 \mod p^{3r-\epsilon_p}.
\end{align*}
This finishes the proof.
\end{proof}

\section{Proof of \Cref{thm: thm 4.34}}\label{section: framed 3-functions}

The present section is dedicated to the proof of \Cref{thm: thm 4.34}.
In the introduction, we gave a short overview of the proof.
As a starting point, we recall the Integrality of Framing Theorem \cite[Thm. 8]{walcher16}.
For the proper use of the Bell transformations above, fix an embedding $K\hookrightarrow \mathbb{C}$.

\begin{theorem}[Framing preserves $2$-integrality]\label{integrality of framing}
The two maps
\begin{align} \label{eq: integrality of framing for Phi+}
\Phi^+\colon \mathbb{Z}\left[ D^{-1} \right] \times \mathcal{S}^2(K | \mathbb{Q})_{\{2\}}
\rightarrow \mathcal{S}^2(K | \mathbb{Q})_{\{2\}},
\end{align}
and
\begin{align} \label{eq: integrality of framing for Phi-}
\Phi^-\colon \mathbb{Z} \times \mathcal{S}^2(K | \mathbb{Q})
\rightarrow \mathcal{S}^2(K | \mathbb{Q}),
\end{align}
are well defined.
Furthermore, $\Phi^+$ defines a group action of the additive group $(\mathbb{Z}\left[ D^{-1} \right],+)$ on $\mathcal{S}^2(K|\mathbb{Q})_{\{2\}}$, while $\Phi^-$ defines a group action of the additive group $( \mathbb{Z},+ )$ on $\mathcal{S} ^2 ( K | \mathbb{Q} )$.
\end{theorem}

\begin{proof}
The proof is due to the work of A. Schwarz, V. Vologodsky and J. Walcher in \cite{walcher16}.
An analogue statement for \textit{fractional framing} is given by \Cref{thm: 2-integrality of fractional framing}, from which \Cref{integrality of framing} follows by setting $\sigma= \rho=1$ therein.
\end{proof}

Note, that $\Phi^-$ satisfies the local $2$-function property even at $p=2$ due to the sign convention, which is not preserved by $\Phi^+$.
However, $\Phi^-$ does not seem to preserve 3-integrality at $p=2$ even for $V \in \mathcal{S}^2_{\mathrm{rat}}(K|\mathbb{Q})$.
Recall, that the coefficients of such a rational $V \in \mathcal{S}^2$  are periodic, as a consequence of \cite[Thm1.2]{mue20}, i.e. there is a (minimal) number $N\in \mathbb{N}$, called \textit{periodicity of $V$}, such that $[V(z)]_n = [V(z)]_{N+m}$ for all $n\in \mathbb{N}$.
Furthermore, 3-integrality also fails for $p=3$ by a $3$-order of $1$ and for all primes $p$ that ramify in $K|\mathbb{Q}$ and which divide the periodicity of $V$.
There are several reasons listed here:
\begin{enumerate}
\item
For a given rational function $V \in \mathcal{S}_\mathrm{rat}^2(K|\mathbb{Q})$ let $N$ denote the periodicity of $V$ and let $S$ be the set of primes dividing $N$.
As an implicit corollary to \cite[Thm. 1.2]{mue20} we obtain that $V $ is also an element in $\mathcal{S}^\infty (K|\mathbb{Q})_S$.
Therefore, for an unramified prime in $K|\mathbb{Q}$ which does not divide $N$ we have the equality $\frob_\mathfrak{p} \left( [V(z)]_n \right) =   [V(z)]_{pn} $, while generally, $\frob_\mathfrak{q}\left( [V(z)]_n \right) \neq [V(z)]_{qn}$ for all primes $q\mid N$, $\mathfrak{q}\mid (q)$ in $\mathcal{O}$.
\item
The Wolstenholme type congruences \Cref{lemma: general wolstenholme} does only permit weaker $p$-adic estimations for $p=2,3$, than for $p\geq 5$.
Also, it depends on a periodic sequence $(a_n)_{n\in\mathbb{N}} \in K^\mathbb{N}$ of periodicity, say, $N\in \mathbb{N}$ (effectively, this is the same $N$ as above and $a_n = [V(z)]_n$).
Because of that, these congruences are additionally weaker for those $p$ dividing $N$ by a $p$-order of $\max\{\ord_2(N), \ord_2(N+2)\}$, $\ord_p(N)$ if $p$ equals to $2$, or greater than $2$, respectively.
\item
The $p$-adic approximation of $e^p$ up to the $p$-power of $3$ gives an additional summand for the primes $2$ and $3$.
\begin{align*}
e^p\equiv\begin{cases}
1+ p + \frac{p^2}2 \mod p^3, & \text{for $p\geq 5$},\\
1+p + \frac{p^2}2 + \frac{p^3}6 \mod p^3, & \text{for $p\in \{2,3\}$}.
\end{cases}
\end{align*}
Since $\Phi^{+/-}$ are implicitely defined by concatenation with the exponential power series $\exp$, illustrated by the functional equations \cref{eq: framing op phi+} and \cref{eq: framing op phi-}, this contributes to the failure of the $3$-itegrality at $p\in \{2,3\}$.
\end{enumerate}
By \Cref{integrality of framing}, the expression
\begin{align}\label{eq: p-adic integer estimation}
\frac2{p^2n^2}\cdot \left( \frob_\mathfrak{p}(a^+_n)- a^+_{pn} \right)
\end{align}
is a $\mathfrak{p}$-adic integer for all $n\in \mathbb{N}$ and all unramified primes $p$, even for $p=2$.
The following lemma gives further estimations of \eqref{eq: p-adic integer estimation}.

\begin{lemma}\label{3-function characterization}
Let $V \in \mathcal{S}^3(K|\mathbb{Q})$ and $\nu \in \mathbb{Z}[D^{-1}]$.
Denote by $a_n = [ V (z) ]_n$ and $a^+_n=\left[ V^{(+, \nu )}(z)\right]_n$ the $n$-th coefficient of $V(z)$ and $V^{(+,\nu)}(z)$, respectively.
Then we have for all (unramified) primes $p$ and prime ideals $\mathfrak{p}$ dividing $(p)$ and for all $n \in \mathbb{N}$ -- except for the case where $p=2$ and $\ord_2(n)=0$ -- the congruence
\begin{align}\label{eq: 3-function condition}
\frac2{p^2n^2}\cdot \left( \frob_\mathfrak{p}(a^+_n)- a^+_{pn} \right) &\equiv \nu \left[ \delta \left( \frob_p V(z^p) + V(z) \right) \cdot \left( \frac{\exp( \nu \smallint V(z) )} z \right)^{pn} \times \right. \nonumber\\
&\left. \hspace{2em} \times \vphantom{\left( \frac{\exp( \nu \smallint V(z) )} z \right)^{pn}} \smallint \hspace{-0.25em}\,^3 \left( \frob_p V(z^p) - V(z)\right) \right]_0 \hspace{-1em}\mod p^{ \ord_p (pn) - \delta_{3,p}} \mathcal{O}_\mathfrak{p}.
\end{align}
Note, that for the exceptional case $p=2$ and $\ord_2(n)$, by \Cref{integrality of framing} we only have $\frob_2\left( a^+_n \right) - a_{2n}^+ \equiv 0 \mod 2\mathcal{O}_{(2)}$, which is already in accordance with the the estimation given in \Cref{integrality of framing}.
\end{lemma}

\begin{proof}
We will consequently exclude the case $p=2$ and $\ord_2(n)=0$ in the following without necessarily mentioning it.
Let $p$ be an unramified prime in $K$.
As in the proof of \Cref{integrality of framing} we will write
\begin{align*}
X(z) := \frob_\mathfrak{p} V\left(z^p\right) - V(z).
\end{align*}
Then we obtain
\begin{align*}
\frob_\mathfrak{p} ( a_n^+ ) - a_{pn}^+ &= \frac1\nu \left[ \frac{ \exp( \nu n p \smallint V(z)) }{z^{pn}} \left( \exp \left( \nu n p \smallint X(z) \right) - 1 \right) \right]_0\\
&= \frac1\nu \left[ \frac{ \exp( \nu n p \smallint V(z)) }{z^{pn}} \sum_{k=1}^\infty  \frac{(\nu n p)^k}{k!} (\smallint X(z))^k \right]_0.
\end{align*}
We find for $k\geq 4$ and $p\geq 3$
\begin{align*}
\mathbb{Z} \ni \ord_p \left( \frac{(pn)^k}{k!} \right) &\geq k(\ord_p(n)+1) - \frac{k-1}2 \\
&= k \left( \ord_p (n) + \frac12 \right) + \frac12 \geq 4 \ord_p (n) + \frac52 > 3 (\ord_p(n) + 1 ) - 1.
\end{align*}
And therefore, $\ord_p\left( \frac{(pn)^k}{k!} \right) \geq 3 \ord_p(pn)$.
For $p=2$ we assume $\ord_2(n)\geq 1$, then for $k\geq 4$
\begin{align*}
\ord_2\left( \frac{(2n)^k}{k!} \right) = k \ord_2(n) + S_2(k)\geq 3 \ord_2(n) +2 = 3(\ord_2(n)+1)-1.
\end{align*}
For $k=3$ we still have
\begin{align*}
\ord_p\left(\frac{(pn)^3}{3!}\right) =
\begin{cases}
3 \ord_p(pn),& \text{if $p\geq 5$},\\
3 \ord_p(pn)-1,&\text{if $p \in \{ 2 , 3 \}$}.
\end{cases}
\end{align*}
Therefore, we obtain for $p\geq 5$,
\begin{align}\label{eq: framed 3-function p>3}
\frob_\mathfrak{p} ( a_n^+ ) - a_{pn}^+ &\equiv
np\left[ \frac{ \exp( \nu n p \smallint V(z)) }{z^{pn}} \right. \times \nonumber\\
&\hspace{0.5em} \times  \left. \vphantom{\frac{ \exp( \nu n p \smallint V(z)) }{z^{pn}}}\left( \smallint X(z) + \frac{\nu np}2 (\smallint X(z))^2\right) \right]_0\mod p^{3\ord_p(pn) }\mathcal{O}_\mathfrak{p},
\end{align}
and for $p\in \{2,3\}$, (again, except for the case where $p=2$ and $\ord_2(n)=0$)
\begin{align}\label{eq: framed 3-function p<5}
\frob_\mathfrak{p} ( a_n^+ ) - a_{pn}^+ &\equiv \left[ \frac{ \exp( \nu n p \smallint V(z)) }{z^{pn}} \times \right. \nonumber\\
&\hspace{1em}\left. \times \left( np \smallint X(z) + \frac\nu2 (np)^2 ( \smallint X(z) )^2 \vphantom{\frac88}\right) \right]_0 \mod p^{3\ord_p(pn)-1} \mathcal{O}_\mathfrak{p}.
\end{align}
We will compute the expressions given in \eqref{eq: framed 3-funct} and \eqref{eq: framed 3-funct p.2} separately.
\begin{align} \label{eq: framed 3-funct}
\left[ \frac{\exp ( \nu n p \smallint V(z) )}{z^{pn}} \smallint X(z)\right]_0 &\mod p^{2\ord_p(pn)}\mathcal O_\mathfrak{p}, &\text{for all primes $p$},\\
\left[ \frac{ \exp ( \nu n p \smallint V(z) )}{z^{pn}} (\smallint X(z))^2 \right]_0 &\mod p^{\ord_p(pn)} \mathcal O_\mathfrak{p}, &\text{for all primes $p$}.\label{eq: framed 3-funct p.2}
\end{align}
In the following, we will write $F(z)= \frac{\exp(\nu np \smallint V(z))}{z^{pn}}$.
We have for all primes $p$
\begin{align*}
\delta^2F(z) &= \delta^2\left(z^{-pn}\exp(\nu p n \smallint V(z))\right)\\
&=\delta\left(-pn z^{-pn} \exp(\nu p n \smallint V(z))+\nu p n V(z) z^{-pn} \exp(\nu p n \smallint V(z)) \right)\\
&= pn\cdot \delta\left((\nu  V(z)-1)F(z)\right)\\
&= pn \nu \cdot \delta V(z) \cdot F(z) + pn(\nu V(z)-1) \cdot \delta F(z)\\
&= pn \nu \cdot \delta V(z) \cdot F(z) + (pn)^2 (\nu V(z)-1)^2 F(z)\\
&\equiv pn\nu \cdot \delta V(z) \cdot F(z) \mod p^{2\ord_p( p n )} \mathcal{O}_\mathfrak{p}\llbracket z \rrbracket.
\end{align*}
Therefore, by using the fact that $\smallint\hspace{-0.25em}\,^3 X(z) \in z \mathcal{O}_\mathfrak{p} \llbracket z \rrbracket$ (for all $p$, which is equivalent to saying $V\in \mathcal{S}^3(K|\mathbb{Q})$), partial integration (see \cref{intbyparts}) applied to \eqref{eq: framed 3-funct} gives us
\begin{align*}
\left[F(z) \cdot \smallint X(z)\right]_0
&=\left[\delta ^2F(z) \cdot  \smallint\hspace{-0.25em}\,^3 X(z)\right]_0\\
&\equiv pn\nu  \left[\delta V(z) \cdot F(z) \cdot \smallint\hspace{-0.25em}\,^3 X(z)
\right]_0 \mod p^{2\ord_p(pn)}\mathcal{O}_\mathfrak{p}.
\end{align*}
Furthermore, \eqref{eq: framed 3-funct p.2} for $p>2$ becomes
\begin{align*}
\left[F(z) (\smallint X(z))^2\right]_0
&=\left[ \delta^2 (F(z) \cdot \smallint X(z)) \cdot \smallint \hspace{-0.25em}\,^3 X(z)\right]_0\\
&\hspace{-4em}=\left[ (\delta^2 F(z) \cdot \smallint X(z) + 2 \cdot \delta F(z) \cdot X(z) + F(z) \cdot \delta X(z)) \cdot \smallint\hspace{-0.25em}\,^3 X(z) \right]_0 \\
&\hspace{-4em}\equiv \left[ (pn\nu\cdot \delta V(z) \cdot \smallint X(z) + \delta X(z) )\cdot F(z) \cdot \smallint\hspace{-0.25em}\,^3 X(z)
\right]_0 \mod p^{\ord_p(pn)}\mathcal{O}_\mathfrak{p}\\
&\hspace{-4em}\equiv
\left[ \delta X(z) \cdot F(z) \cdot \smallint\hspace{-0.25em}\,^3 X(z)
\right]_0 \mod p^{\ord_p(pn)}\mathcal{O}_\mathfrak{p}.
\end{align*}
Therefore, inserting \eqref{eq: framed 3-funct} and \eqref{eq: framed 3-funct p.2}, for $p\geq 5$, into \cref{eq: framed 3-function p>3}, we obtain
\begin{align*}
\frob_\mathfrak{p} ( a_n^+ ) - a_{pn}^+ \hspace{-4em}&\\
&\equiv
\nu (np)^2 \left[ \delta V(z) \cdot F(z) \cdot \smallint\hspace{-0.25em}\,^3 X(z)
\right]_0 + \\
&\hspace{5em} + \frac\nu2(np)^2\left[\delta X(z) \cdot F(z) \cdot \smallint\hspace{-0.25em}\,^3 X(z) \right]_0 \mod p ^{3\ord_p(pn)} \mathcal{O}_\mathfrak{p}\\
&=\frac\nu2 (np)^2 \left[\delta\left( 2V(z) + X(z) \right) \cdot F(z) \cdot \smallint\hspace{-0.25em}\,^3 X(z) \right]_0 \mod p^{3\ord_p(pn)} \mathcal{O}_\mathfrak{p}\\
&=\frac\nu2 (np)^2\left[\delta (\frob_pV(z^p) + V(z)) \cdot F(z) \cdot \smallint\hspace{-0.25em}\,^3 X(z)\right]_0 \mod p^{3\ord_p(pn)} \mathcal{O}_\mathfrak{p},
\end{align*}
which proves \cref{eq: 3-function condition} for $p\geq 5$.
For $p\in \{2,3\}$, \cref{eq: framed 3-function p<5} becomes
\begin{align*}
\frob_\mathfrak{p} ( a_n^+ ) - a_{pn}^+\hspace{-4em}&\\
&\equiv \frac\nu2 (np)^2\left[\delta (\frob_pV(z^p) + V(z)) \cdot F(z) \cdot \smallint\hspace{-0.25em}\,^3 X(z)\right]_0 \mod p^{3\ord_p(pn)-1} \mathcal{O}_\mathfrak{p}.
\end{align*}
As stated in \cref{eq: 3-function condition} for $p\in \{2,3\}$.
\end{proof}

It is very tedious to check whether $V^{(+/-,\nu)}(z)$ satisfies the local $3$-function property for a given prime $p$ by using \cref{eq: 3-function condition} explicitly.
However, for $V\in \bigcap_{s=1}^\infty \mathcal{S}^s(K|\mathbb{Q})= \mathcal{S}^\infty (K|\mathbb{Q})$ we may simplify \cref{eq: 3-function condition}.
This is the statement of the following \Cref{cor:3-function characterization}.

\begin{corollary}\label{cor:3-function characterization}
Let $V(z)\in \mathcal{S}^\infty (K|\mathbb{Q})$ and $\nu \in \mathbb{Z}\left[D^{-1}\right]$.
Denote by $a_n^+$ the $n$-th coefficient of $V^{(+,\nu)}(z)$ for all $n\in \mathbb{N}$.
Then for all (unramified) primes $p$ and all $n\in \mathbb{N}$ we have -- except for the case where $p=2$ and  $\ord_2(n)=0$ -- the congruence
\begin{align*}
\frac2{p^2n^2}\cdot \left( \frob_\mathfrak{p}(a^+_n)- a^+_{pn} \right) &\equiv \\
&\hspace{-10em} \nu \left[ V(z) \cdot \left( \frac{\exp(\nu \smallint V(z))}{z}\right)^{pn} \cdot \smallint\hspace{-0.25em}\,^2 \left(\frob_pV(z^p) - V(z)\right)\right]_0 \mod p^{\ord_p(pn)-\delta_{3,p}} \mathcal{O}_\mathfrak{p}.
\end{align*}
\end{corollary}

\begin{proof}
Let $p$ be an unramfied prime in $K|\mathbb{Q}$ and fix $n \in \mathbb{N}$.
As in the proof of \Cref{integrality of framing} we will write 
\begin{align*}
X(z) := \frob_\mathfrak{p} V\left(z^p\right) - V(z) \quad \text{and}\quad
F(z):= z^{-pn}\exp(\nu pn \smallint V(z)).
\end{align*}
By assumption, $\smallint\hspace{-0.25em}^s X(z)\in z\mathcal{O}_p\llbracket z \rrbracket$ for all $s \in \mathbb{N}$.
Equivalently, $[X(z)]_{pn}=0$ for all $n\in \mathbb{N}$.
Let $s=\ord_p(n)+3$. Then
\begin{align*}
\left[ \delta \left( \frob_\mathfrak{p} V ( z^p ) + V(z) \right) \cdot F(z) \cdot
\delta^{\ord_p(n)}\smallint\hspace{-0.25em}\,^{\ord_p(n)+3} X(z) \right]_0\\
&\hspace{-21em}=(-1)^{\ord_p(n)}\left[ \delta^{\ord_p(n)} \left( \delta\left( \frob_\mathfrak{p} V(z^p) + V(z) \right)\cdot F(z) \right)\cdot\smallint\hspace{-0.25em}\,^{\ord_p(n)+3} X(z)
\right]_0.
\end{align*}
Note that $\delta F(z)\equiv 0\mod p^{\ord_p(pn)}$, therefore
\begin{align*}
\delta^{\ord_p(n)}\left( 
\delta \left( \frob_\mathfrak{p} V(z^p) + V(z) \right) \cdot F(z) \right) &\\
&\hspace{-10em}\equiv\delta^{\ord_p(n)+1} \left( \frob_pV(z^p)+V(z)\right)\cdot F(z) \mod p^{\ord_p(pn)} z\mathcal{O}_\mathfrak{p}\llbracket z \rrbracket\\
&\hspace{-10em} \equiv \delta^{\ord_p(pn)} V(z) \cdot F(z) \mod p^{\ord_p(n)+1} z\mathcal{O}_\mathfrak{p}\llbracket z \rrbracket\\
&\hspace{-10em}\equiv \delta^{\ord_p(pn)} (V(z) F(z)) \mod p^{\ord_p(n)+1} z\mathcal{O}_\mathfrak{p}\llbracket z \rrbracket.
\end{align*}
Therefore, by \Cref{3-function characterization}, we have
\begin{align*}
-\frac2{p^2n^2} \cdot \left( \frob_\mathfrak{p}(a_n^+) - a_{pn}^+ \right) \equiv \nu \left[ V(z) \cdot F(z) \cdot \smallint\hspace{-0.25em}\,^2 X(z)\right]_0
\mod p^{\ord_p(pn) - \delta_{3,p}} \mathcal{O}_\mathfrak{p},
\end{align*}
as stated.
\end{proof}

\begin{lemma}\label{lemma: a trivial lemma}
For all $V\in \mathcal{S}^1(K|\mathbb{Q})$ and $r\in \mathbb{N}$ and unramified primes $p$ in $K|\mathbb{Q}$ we have
\begin{align*}
\exp(p^r \smallint (\frob_\mathfrak{p} V (z^p) - V(z))) \in 1 + p^{r}z \mathcal{O}_\mathfrak{p} \llbracket z \rrbracket.
\end{align*}
\end{lemma}

\begin{proof}
Write $X(z)= \frob_\mathfrak{p}V(z^p)- V(z)$. Since $V \in \mathcal{S}^1(K|\mathbb{Q})$ we have $\smallint X(z) \in z\mathcal{O}_p\llbracket z \rrbracket$ for all unramified primes $p$ in $K|\mathbb{Q}$.
In particular, the statement follows if
\begin{align*}
	\exp(p^r \widetilde X(z)) \in 1 + p^r z  \mathcal{O}_p\llbracket z \rrbracket
\end{align*}
for any $\widetilde X\in \mathcal{O}_p\llbracket z \rrbracket$.
We have
\begin{align*}
\exp(p^r \widetilde X (z)) = 1+ \sum_{k=1}^\infty \frac{p^{rk}}{k!} \widetilde X(z)^k.
\end{align*}
Then
\begin{align*}
\ord_p\left(\frac{p^{rk}}{k!}\right)= rk- \frac{k- S_p(k)}{p-1}\overset{p\geq 2}\geq rk- k + S_p(k)
\overset{S_p(k)\geq 1}\geq (r-1)k + 1 \overset{k\geq 1}\geq r,
\end{align*}
from which the statement follows.
\end{proof}

Finally, we put the pieces together:

\begin{proof}[Proof of \Cref{thm: thm 4.34}]
Let $V\in \mathcal{S}_{\mathrm{rat}}^2(K|\mathbb{Q})$ and $\nu \in \mathbb{Z}\left[ D^{-1} \right]$ and $N$ the periodicity of $V$.
Furthermore, let $S$ be the set of primes dividing $N$.
Fix an unramified prime $p$ in $K|\mathbb{Q}$ and $n\in \mathbb{N}$, such that $p\not\in S$.
Let $a_m:= [V(z)]_m$ and write $X(z)= \mathrm{Frob}_\mathfrak{p} V (z^p)- V(z)$.
Hence, since $\mathcal{S}_{\mathrm{rat}}^2(K|\mathbb{Q})\subset \mathcal{S}^\infty(K|\mathbb{Q})_{S}$
\begin{align*}
	X(z)= - \sum_{\substack{k=1 \\ p\nmid k}}^\infty a_k z^k.
\end{align*}
By \Cref{lemma: a trivial lemma} and since $ \frac{\nu n}{p^{\ord_p(n)}}V \in \mathcal{S}^1(K|\mathbb{Q})$, we obtain
\begin{align*}
\exp(\nu np \smallint V(z))
&=\exp (-\nu np \smallint X(z))\exp(\nu n p \smallint (\mathrm{Frob}_p V(z^p)))\\
&\equiv \exp(\nu n p \smallint (\mathrm{Frob}_p V(z^p))) \mod p^{\ord_p(n)+1} \mathcal{O}_\mathfrak{p}\\
&= \exp\left(\nu n \sum_{k=1}^\infty \frac{a_{pk}}k z^{pk}\right).
\end{align*}
Let us denote $\displaystyle \exp \left( \sum_{k=1}^\infty \frac{a_{pk}}k z^k \right) = 1+ \sum_{k=1}^\infty y_k z^k=Y(z)$ and
\begin{align*}
Y(z)^{\nu n} = \exp\left(\nu n \sum_{k=1}^\infty \frac{a_{pk}}k z^k\right)
&= \tilde Y(z).
\end{align*}
By Dwork's Integrality \Cref{characterization of 1-functions}, we have
\begin{align*}
\widetilde Y (z), Y(z)  \in \mathcal{O}_\mathfrak{p} \llbracket z \rrbracket.
\end{align*}
Set $\tilde y_m = \left[\widetilde Y (z)\right]_m$ for all $m\in \mathbb{N}_0$.
Note that by \Cref{lemma: final countdown},
\begin{align*}
\tilde y_m \equiv 0 \mod p^{\max\{ 0, \ord_p(\nu n) - \ord_p(m) \}}\mathcal{O}_\mathfrak{p}.
\end{align*}
Then we compute the expression given in \Cref{cor:3-function characterization} explicitly 
\begin{align}\label{eq: crucial reduction executed}
\frac2{p^2n^2}(\mathrm{Frob}_\mathfrak{p}(a_n^+) - a_{pn}^+)
&\equiv  - \nu \left[ \frac{V(z)}{z^{pn}}  \widetilde Y(z^p) \smallint\,\hspace{-0.25em}^2 X(z)\right]_0 \mod p^{\ord_p(n)+1 - \delta_{p,3}} \mathcal{O}_\mathfrak{p}\nonumber\\
&=
\nu \sum_{k=1}^\infty \sum_{\substack{\ell = 1 \\ p\nmid \ell}}^\infty \sum_{m=0}^\infty \tilde y_m \frac{a_ka_\ell}{\ell^2}\left[ z^{p(m-n) + k + \ell} \right]_0\nonumber\\
&= 
\nu \sum_{m=0}^n \tilde y_m \sum_{\substack{\ell = 1 \\ p\nmid \ell}}^{p(n-m)}\frac{a_{p(n-m)-\ell}a_\ell}{\ell^2},
\end{align}
where for the last step, we used $\displaystyle \left[z^{p(m-n)+k+\ell}\right]_0 = \delta_{k,p(n-m)-\ell}$.
We need to compute 
\begin{align}\label{eq: countdown}
x(m)=\ord_p \left( \nu \tilde y_m \sum_{\substack{ \ell =1 \\ p\nmid \ell} }^{p(n-m)} \frac{a_{p(n-m)-\ell}a_\ell}{\ell^2} \right).
\end{align}
By \Cref{lemma: general wolstenholme} and \Cref{lemma: final countdown} and respecting the $p$-adic estimation used in the calculation given in \cref{eq: crucial reduction executed}, we obtain
\begin{align*}
x(m) &\geq\\
&\hspace{-2em} \min \left\{ \ord_p(pn) - \delta_{p,3}, \max\{ 0, \ord_p(n) - \ord_p(m) \} + \max\{0, \ord_p(p(n-m)) - \gamma_p\} \right\},
\end{align*}
where $\gamma_p$ is given as in \Cref{lemma: general wolstenholme}.
\begin{itemize}
\item
For $\ord_p (n) \geq \ord_p(m)$ and $\gamma_p \leq \ord_p(n-m)+1$ we have
\begin{align*}
x(m) &\geq \min\left\{ \ord_p(n)+1 - \delta_{p,3},\ord_p(n)- \ord_p(m)+ \ord_p(m)+1 - \gamma_p \right\} \\
&=\ord_p(n)+1 -\gamma_p\geq 0.
\end{align*}

\item
For $\ord_p(n)\geq \ord_p(m)$ and $\gamma_p > \ord_p(n-m) + 1$, then $-\ord_p(m) > 1-\gamma_p$ and therefore
\begin{align*}
x(m) &\geq \min \left\{ \ord_p(n)+1 -\delta_{p,N}, \ord_p(n)- \ord_p(m) + \ord_p(m) + 1 - \gamma_p \right\}\\
&= \ord_p(n)+1 - \gamma_p\geq 0.
\end{align*}

\item
For $\ord_p(n) < \ord_p(m)$ and $\gamma_p \leq \ord_p(n-m) + 1$, we have
\begin{align*}
x(m) &\geq \min \left\{ \ord_p(n)+1 -\delta_{p,N}, \ord_p(n) + 1 - \gamma_p \right\}\\
&= \ord_p(n)+1 - \gamma_p\geq 0.
\end{align*}

\item
For $\ord_p(n) < \ord_p(m)$ and $\gamma_p > \ord_p(n-m) + 1$, we have
\begin{align*}
x(m) &\geq \min \left\{ \ord_p(n)+1 -\delta_{p,N}, 0 \right\}\\
&= 0 > \ord_p(n-m)+1 - \gamma_p= \ord_p(n)+1-\gamma_p.
\end{align*}

\end{itemize}

Therefore, for $x := \min\{x(m)\,|\, m \in \{0,...,n\}\}$, we have
\begin{align*}
x \geq
\max\{0, \ord_p(n) + 1 - \gamma_p\}.
\end{align*}
Hence,
\begin{align*}
\mathrm{Frob}_\mathfrak{p}(a_n^+)- a_{pn}^+ \equiv
0 \mod p^{2( \ord_p(n) + 1) - \delta_{2,p} + \max\{0, \ord_p(n) + 1 - \gamma_p\} } \mathcal{O}_\mathfrak{p},
\end{align*}
as stated.
In particular, for $p \geq 5$ unramified in $K|\mathbb{Q}$, that does not divide $N$, we have (in this case, $\gamma_p=0$)
\begin{align*}
\mathrm{Frob}_\mathfrak{p} (a_n^+)- a_{pn}^+ \equiv
0 \mod p^{3( \ord_p(n) + 1)} \mathcal{O}_\mathfrak{p}.
\end{align*}
Nonetheless, for
\begin{align*}
 C = 2 \cdot \prod_{p\text{ prim}} p^{\gamma_p},
\end{align*}
we have $C \cdot V^{(+,\nu )}(z) \in \mathcal{S}^3(K|\mathbb{Q})_S$, and therefore,
$V^{(+,\nu)}(z) \in \overline{\mathcal{S}}^3 (K|\mathbb{Q})_S$.
\end{proof}

\section{Fractional Framing and the Proof of \Cref{thm: intro 2-integrality of fractional framing}}\label{section: fractional framing and examples}

In this section we will introduce the notion of fractional framing.
For $\nu\in \mathbb{Q}$ and $V\in \mathcal{S}^2(K|\mathbb{Q})$, $V^{(-,\nu)}$ fails to to fulfill the local 2-function property precisely at those $p$ such that $\ord_p(\nu)<0$.
This can be fixed by applying the Cartier operator $\mathscr{C}_\sigma$ to $V^{(-,\nu)}$, where $\ord_p(\sigma \nu)\geq 0$.
This is referred to as \textit{fractional framing}.

\begin{theorem}\label{thm: 2-integrality of fractional framing}
Let $V\in \mathcal{S}^2(K|\mathbb{Q})$ and $\nu \in \mathbb{Q}$ and $\rho,\sigma \in \mathbb{N}$, such that $\gcd(\rho,\sigma)=1$.
Then, if $\nu\frac\sigma\rho \in \mathbb{Z}\left[ D^{-1} \right]$,
\begin{align*}
\frac1\sigma \varepsilon^{(2)}_\rho \left( \mathscr{C}_\sigma \left( \Phi^-(\nu,V) \right)  \right) \in \mathcal{S}^2 (K | \mathbb{Q} )
\end{align*}
and, if $\nu\frac\sigma\rho \in \mathbb{Z}$,
\begin{align*}
\frac1\sigma \varepsilon^{(2)}_\rho \left( \mathscr{C}_\sigma \left( \Phi^+(\nu,V) \right)  \right) \in \left( \mathcal{S}^2 (K | \mathbb{Q} )_{\{2\}} \cap \overline{\mathcal{S}}^2(K|\mathbb{Q})\right).
\end{align*}
\end{theorem}

\begin{proof}
The proof we are presenting here goes follows the same arguments and steps as the proof of \Cref{integrality of framing}. As above, we assume $\nu\neq 0$.

\noindent We write $\widetilde V = \frac1\sigma \varepsilon^{(2)}_\rho \left( \mathscr{C}_\sigma \left( \Phi^-(\nu,V) \right)  \right)$ and $\tilde a_n^- := \left[ \widetilde V(z)\right]_n$ for all $n \in \mathbb{N}$.
We have
\begin{align*}
\tilde a_n^- = \frac{\rho^2}\sigma \left[ \mathscr{C}_\sigma V^{(-,\nu)}(z) \right]_{\nicefrac n\rho}
= \frac{\rho^2}\sigma a_{\nicefrac{\sigma n}\rho}^-,
\end{align*}
with the understanding that $\tilde a_n^-= 0$, whenever $\rho \nmid n$.
Then
\begin{align*}
\frob_\mathfrak{p}(\tilde a^-_n) - \tilde a^-_{pn}=
\begin{cases}
0,& \text{if $\rho \nmid pn$},\\
- \frac{\rho^2}\sigma a^-_{\nicefrac{\sigma pn}\rho},& \text{if $\rho \mid pn$, but $\rho\nmid n$},\\
\frac{\rho^2}\sigma \left( \frob_\mathfrak{p}\left( a_{\nicefrac{\sigma n}\rho}^- \right) -  a_{\nicefrac{\sigma pn}\rho}^- \right),& \text{if $\rho\mid n$}.
\end{cases}
\end{align*}
In the first two cases, the local 2-function property at the prime $p$ is trivially satisfied.
For $\rho\mid n$, we still need to check
\begin{align*}
\frob_\mathfrak{p}\left( a^-_{\nicefrac{\sigma n}\rho} \right) - a_{\nicefrac{\sigma pn}\rho}^- \equiv 0 \mod p^{2(\ord_p(n)+1- \ord_p(\rho))+\ord_p(\sigma)}\mathcal{O}_\mathfrak{p}.
\end{align*}
In the following, we will assume $\ord_p(\rho)\leq \ord_p(n)$, which is an implementation of the condition $\rho \mid n$.
\begin{itemize}
\item[\textit{Case 1:}] $p\geq 3$.
Let $p$ be a prime number unramified in $K|\mathbb{Q}$ greater than $3$.
Recall that $a^-_{\nicefrac{\sigma n}\rho}= (-1)^{\nu \frac{\sigma n}\rho} a_{\nicefrac{\sigma n}\rho}^+$.
As before, we write
\begin{align*}\smallint\hspace{-0.25em}\,^2X(z) = \smallint\hspace{-0.25em}\,^2 \left( \frob_\mathfrak{p} V\left(z^p\right) - V(z) \right).
\end{align*}
Then by the same $p$-adic estimation as given in \textit{Case 1} of the proof of \Cref{integrality of framing},  we have
\begin{align*}
\frob_\mathfrak{p}( a_{\nicefrac{\sigma n}\rho}^+) - a_{\nicefrac{\sigma pn}\rho}^+
&= \left[ \frac{\exp \left( \nu \frac\sigma\rho n p \smallint V(z) \right)}{z^{\nicefrac{\sigma pn }\rho}} \sum_{k=1}^\infty (\sigma\nu)^{k-1} \frac\sigma{k!} \left( \frac{np}\rho \right)^k \left( \smallint X(z) \right)^k \right]_0.
\end{align*}
Using $\ord_p(\sigma \nu) \geq 0$ and $\rho\mid n$ we obtain for $k\geq 2$
\begin{align*}
\ord_p\left( (\sigma\nu)^{k-1} \frac\sigma{k!} \left( \frac{np}\rho \right)^k \right)
&\geq 2 \ord_p\left( \frac n\rho p \right) + \ord_p(\sigma).
\end{align*}
Therefore,
\begin{align*}
\frob_\mathfrak{p}( a_{\nicefrac{\sigma n}\rho}^+) - a_{\nicefrac{\sigma pn}\rho}^+\hspace{-8em}&\\
&\equiv \frac\sigma\rho pn \left[ \frac{\exp\left( \nu \frac\sigma\rho p n \smallint V(z) \right)}{z^{\nicefrac{pn\sigma}\rho}} \cdot \smallint X(z) \right]_0 \mod p^{2(\ord_p(n)+1 - \ord_p(\rho)) + \ord_p(\sigma)} \mathcal{O}_\mathfrak{p}\\
&= -\sigma \left( \frac{pn}\rho \right)^2\cdot \left[ (\sigma \nu V(z) - \sigma) \cdot \left( \frac{\exp(\nu \smallint V(z))}z \right)^{\nicefrac{pn\sigma}\rho} \hspace{-0.25em} \cdot \smallint\hspace{-0.25em}\,^2 X(z) \right]_0.
\end{align*}
Since $\ord_p(\nu\sigma)\geq 0$, the expression in $[-]_0$ is a $p$-adic integer.
Therefore,
\begin{align*}
\frob_\mathfrak{p} \left( a_{\nicefrac{\sigma n}\rho}^+ \right) - a_{\nicefrac{\sigma pn}\rho}^+ &=0 \mod p^{2(\ord_p(n)+1-\ord_p(\rho))+\ord_p(\sigma)} \mathcal{O}_\mathfrak{p}.
\end{align*}
\item[\textit{Case 2:}] \textit{$p=2$, and $\ord_2\left(\frac{\sigma n}\rho \nu\right)\geq 1$.}
Then, if $\frac{\sigma n}\rho \in \mathbb{Z}$,
\begin{align*}
\frob_2 \left( a_{\nicefrac{\sigma n}\rho}^- \right) - a_{\nicefrac{2\sigma n}\rho}^-
&= (-1)^{\nicefrac{\nu \sigma n}\rho} \left( \frob_2 \left( a_{\nicefrac{\sigma n}\rho}^+ \right) - a_{\nicefrac{2\sigma n}\rho}^+ \right)
\end{align*}
Therefore, it suffices to check the congruence for $\frob_2 \left( a_{\nicefrac{\sigma n}\rho}^+ \right) - a_{\nicefrac{2\sigma n}\rho}^+$ and we may assume $\frac{\sigma n}\rho \in \mathbb{Z}\left[D^{-1}\right]$.
We have
\begin{align*}
 \frob_2 \left( a_{\nicefrac{\sigma n}\rho}^+ \right) - a_{\nicefrac{2\sigma n}\rho }^+ \hspace{-2em}&\\
 &= \left[ \frac{\exp \left( 2 \nu \frac\sigma\rho n \smallint V(z) \right)}{z^{\nicefrac{2 \sigma n}\rho}}\left( 
\sum_{k=1}^\infty (\sigma \nu)^{k-1} \frac\sigma{k!}\left( \frac{2n}\rho \right)^k \left( \smallint X(z) \right)^k \right)\right]_0.
\end{align*}
For $k\geq 3$ we have
\begin{align*}
\ord_2\left( (\sigma \nu)^{k-1} \frac\sigma{k!}\left( \frac{2n}\rho \right)^k \right)\hspace{-8em}&\\
&= k \ord_2 \left( \frac n\rho \right) + \ord_2(\sigma) + 1 + (k-1) \ord_2 (\nu \sigma) - \ord_2(k!) + S_2(k).
\end{align*}
Recall that $S_2 (k)$ denotes the sum of the digits of $k$ in base $2$.
Using $S_2(k) \geq 1$ for all $k\in \mathbb{N}$, $\ord_2\left( \frac{\sigma n }\rho \nu \right) \geq 1$, and $\ord_2(\nu \sigma)\geq 0$, we obtain
\begin{align*}
\ord_2\left( (\sigma \nu)^{k-1} \frac\sigma{k!}\left( \frac{2n}\rho \right)^k \right)
&\overset{k\geq 3}\geq 2\left( \ord_2\left( \frac n\rho \right) +1 \right) + \ord_2(\sigma).
\end{align*}
Therefore,
\begin{align*}
 \frob_2 \left( a_{\nicefrac{\sigma n}\rho}^+ \right) - a_{\nicefrac{2\sigma n}\rho }^+ &\equiv \frac{2\sigma n}\rho \left[ \frac{\exp \left( 2 \nu \frac\sigma\rho n \smallint V(z) \right)}{z^{\nicefrac{2 \sigma n}\rho}} \times \right. \\
 &\hspace{-2em}\left.\times \left( \smallint X(z) + \nu \frac\sigma\rho n \left( \smallint X(z) \right)^2 \right) \right]_0 \mod 2^{2 \left( \ord_2 \left(\frac n\rho \right) +1 \right) + \ord_2(\sigma)} \mathcal{O}_2.
\end{align*}
What remains to show is
\begin{align}\label{eq: fractional p=2 for Phi+}
 \left[ \frac{\exp \left( 2 \nu \frac\sigma\rho n \smallint V(z) \right)}{z^{\nicefrac{2 \sigma n}\rho}} \smallint X(z)\right]_0  \equiv 0 \mod 2^{\ord_2(n)-\ord_2(\rho)+1} \mathcal{O}_2
\end{align}
and
\begin{align}\label{eq: fractional p=2 for Phi+ p1andahalf}
\nu \sigma \left[ \frac{\exp \left( 2 \nu \frac\sigma\rho n \smallint V(z) \right)}{z^{\nicefrac{2 \sigma n}\rho}} \left( \smallint X(z) \right)^2 \right]_0 \equiv 0 \mod 2 \mathcal{O}_2.
\end{align}
The first summand (a.k.a. \cref{eq: fractional p=2 for Phi+}) vanishes by the same calculation as in the previous case.
Therefore, it remains to show \cref{eq: fractional p=2 for Phi+ p1andahalf}.
Since $x_i \in 2^{2\ord_2(i)} \mathcal{O}_{(2)}$, we have
\begin{align}\label{eq: peven}
(\smallint X(z))^2 = \sum_{i,j=1}^\infty \frac{x_ix_j}{ij} z^{i+j} &=
2 \sum_{\substack{i,j=1 \\ i<j}}^\infty\frac{x_ix_j}{ij}z^{i+j}
+\sum_{i=1}^\infty \frac{x_i^2}{i^2}z^{2i}\nonumber\\
&\equiv \sum_{\substack{i=1\\ i \text{ odd}}}^\infty \frac{x_i^2}{i^2} z^{2i}
\equiv \sum_{\substack{i=1\\ i \text{ odd}}}^\infty x_i^2 z^{2i}
\mod 2 z\mathcal{O}_2 \llbracket z \rrbracket.
\end{align}
Hence,
\begin{align*}
\left[ \frac{\exp \left( 2 \nu \frac\sigma\rho n \smallint V(z) \right)}{z^{\nicefrac{2 \sigma n}\rho}} \left( \smallint X(z) \right)^2 \right]_0 \hspace{-3em}& \\
&\equiv \sum_{\substack{i=1\\ i\text{ odd}}}^\infty x_i^2\left[ \exp \left( 2 \nu \frac\sigma\rho n \smallint V(z) \right)  \right]_{2\left( \frac{\sigma n}\rho - i \right)} \mod 2 \mathcal{O}_2.
\end{align*}
Using \Cref{lemma: final countdown}, we find for all odd $i\in \mathbb{N}$, $i\leq \frac\sigma\rho n$,
\begin{align}\label{eq: pos}
\nu \sigma \left[ \exp\left(2\nu \frac\sigma\rho n \smallint V(z) \right) \right]_{2\left( \frac\sigma\rho n - i\right)}
\equiv 0  \mod 2 \mathcal{O}_2,
\end{align}
since:
\begin{itemize}
\item if $\ord_2\left( \frac{\sigma n}\rho\right)\geq 1$, then $\ord_2\left( \frac{\sigma n}\rho -i \right) = \ord_2(i)=0$ and therefore
\begin{align*}
\ord_2\left( 2 \nu \frac\sigma\rho n \right) - \ord_2 \left( 2 \left( \frac{\sigma n}\rho - i \right) \right)
\geq 2-1 =1.
\end{align*}
\item if $\ord_2 \left( \frac{\sigma n}\rho \right) =0$, then $\ord_2(\nu)\geq 1$ (and therefore, $\ord_2(\nu\sigma)\geq 1$).
In that case, the congruence \cref{eq: pos} is immediately satisfied, since the power series in the brackets $[-]_0$ has $2$-adic integral coefficients by Dwork's Integrality Lemma.
\end{itemize}
Finally, we have
\begin{align*}
\frob_2 \left( a_{\nicefrac{\sigma n}\rho}^+ \right) - a_{\nicefrac{2\sigma n}\rho }^+ \equiv 0 \mod 2^{2(\ord_2(n)+1-\ord_2(\rho)) + \ord_2 (\sigma)} \mathcal{O}_2.
\end{align*}
\item[\textit{Case 3:}] \textit{Let $p=2$, $\ord_2 \left( \nu \frac\sigma\rho n \right) = 0$ and $\nu \frac\sigma\rho\in \mathbb{Z}$.}
First recall that $\gcd(\sigma,\rho)=1$ by definition, $\ord_2(\rho)\leq \ord_2(n)$, since $\rho\mid n$ by assumption, and $\ord_2(\nu \sigma)\geq 0$.
Therefore, we immediately see that $\ord_2\left( \frac n\rho \right) = \ord_2( \nu \sigma )=0$.
Indeed, since we assume $\ord_2\left( \nu\frac \sigma \rho n \right)=0$, we have
\begin{align*}
0 \leq \ord_2 \left( \frac n\rho \right) = - \ord_2 \left(\nu \sigma \right)\leq 0.
\end{align*}
Note that
\begin{align*}
(-1)^{\nicefrac{\nu \sigma n}\rho} \left( \frob_2 \left( a_{\nicefrac{n\sigma}\rho}^- \right) - a_{\nicefrac{2n\sigma}\rho}^- \right)
&= \frob_2 \left( a_{\nicefrac{n\sigma}\rho}^+ \right) + a_{\nicefrac{2n \sigma}\rho}^+.
\end{align*}
Therefore, we need to show
\begin{align*}
\frob_2 \left( a^+_{\nicefrac{n\sigma}\rho} \right) + a^+_{\nicefrac{2n \sigma}\rho} \equiv 0 \mod 2^{2+ \ord_2(\sigma)} \mathcal{O}_2.
\end{align*}
We have
\begin{align*}
\frob_2 \left( a_{\nicefrac{n\sigma}\rho}^+ \right) + a_{\nicefrac{2n\sigma}\rho}^+\hspace{-4em}&\\
&\equiv
\left[ \frac{\exp \left( 2\nu \frac\sigma\rho n \smallint V(z) \right)}{z^{\nicefrac{2\sigma n}\rho}} \left( \frac2\nu + \sum_{k=1}^\infty (\nu \sigma)^{k-1} \frac\sigma{k!} \left( \frac{2n}\rho \right)^k (\smallint X(z))^k \right) \right]_0.
\end{align*}
For all $k\in \mathbb{N}$ such that $k \neq 2^\ell$ for some $\ell \in \mathbb{N}$, we have $\ord_2(k)\geq 2$ and therefore, for such $k$,
\begin{align*}
\ord_2\left( (\nu \sigma)^{k-1} \frac\sigma{k!} \left( \frac{2n}\rho\right)^k \right)
&\geq 2+\ord_2(\sigma).
\end{align*}
Also, note that $\ord_2 \left( \frac{2^{2^\ell}}{(2^\ell)!} \right) = 2^\ell - 2^\ell + 1= 1 $ for all $\ell \in \mathbb{N}$ -- and hence $\frac{2^{2\ell}}{2(2^\ell)!} \equiv 1 \mod 2$ -- we obtain
\begin{align*}
\frob_2 \left( a_{\nicefrac{n\sigma}\rho}^+ \right) + a_{\nicefrac{2n\sigma}\rho}^+ &\equiv 2\sigma \left[ \frac{\exp \left( 2\nu \frac\sigma\rho n \smallint V(z) \right)}{z^{\nicefrac{2\sigma n}\rho}}\times \right. \\ &\hspace{-5em}\times \left. \left( \frac1{\nu\sigma} + \sum_{\ell=0}^\infty  \frac{ (\nu \sigma)^{2^\ell -1} }{2(2^\ell)!} \left( \frac{2n}\rho \right)^{2^\ell} (\smallint X(z))^{2^\ell} \right) \right]_0 \mod 2 ^{2 + \ord_2 (\sigma)} \mathcal{O}_2.
\end{align*}
Note that $\frac1{\nu\sigma} \equiv \frac n\rho \equiv \frac{2^{2^\ell}}{2(2^\ell)!} \equiv 1 \mod 2 \mathbb{Z}_2$, it remains to prove
\begin{align}\label{eq: all we need}
\left[ \frac{\exp \left( 2\nu \frac\sigma\rho n \smallint V(z) \right)}{z^{\nicefrac{2\sigma n}\rho}} \left( 1 + \sum_{\ell=0}^\infty (\smallint X(z))^{2^\ell} \right) \right]_0 \equiv 0 \mod 2 \mathcal{O}_2.
\end{align}
We find for all $\ell\in \mathbb{N}_0$,
\begin{align*}
(\smallint X(z))^{2^\ell} \equiv \sum_{\substack{i=1\\ i \text{ odd}}}^\infty x_i^{2^\ell} z^{2^\ell i}
\mod 2 z\mathcal{O}_2\llbracket z \rrbracket.
\end{align*}
Note that for odd $i\in \mathbb{N}$ (and for $a_{i/2} := 0$ in this case) and since $V$ is in particular an element in $\mathcal{S}^1(K|\mathbb{Q})$, we have
\begin{align*}
x_i^{2^\ell} = (\mathrm{Frob}_2 (a_{i/2}) - a_i)^{2^\ell} = a_i^{2^\ell}\equiv \mathrm{Frob}_2^\ell(a_i) \equiv a_{2^\ell i}\mod 2 \mathcal{O}_2.
\end{align*}
Hence,
\begin{align}\label{eq: 20120714}
\sum_{\ell=0}^\infty (\smallint X(z))^{2^\ell}
&\equiv\sum_{\ell=0}^\infty \sum_{\substack{i=1\\ i \text{ odd}}}^\infty x_i^{2^\ell} z^{2^\ell i}\nonumber\\
&\equiv \sum_{\ell=0}^\infty \sum_{\substack{i=1\\ i \text{ odd}}}^\infty a_{2^\ell i} z^{2^\ell i} 
\equiv \sum_{k=1}^\infty a_k z^k = V(z) \mod 2 \mathcal{O}_2.
\end{align}
Therefore, \cref{eq: all we need} follows from \cref{eq: 20120714} and \cref{eq: partial integration for framing coefficients}
\begin{align*}\left[ \frac{\exp \left( 2 \nu \frac\sigma\rho n \smallint V(z) \right)}{z^{\nicefrac{2\sigma n}\rho}} \left( 1 + \sum_{\ell=0}^\infty \left( \smallint X(z) \right)^{2^\ell} \right) \right]_0 \hspace{-8em}&\\
&\hspace{-1.25em}\overset{\cref{eq: 20120714}}\equiv \left[ \frac{\exp \left( 2 \nu \frac\sigma\rho n \smallint V(z) \right)}{z^{\nicefrac{2\sigma n}\rho}} \left( 1 + V(z) \right) \right]_0 \mod 2\mathcal{O}_2\\
&\equiv \left[ \frac{\exp \left( 2 \nu \frac\sigma\rho n \smallint V(z) \right)}{z^{\nicefrac{2\sigma n}\rho}} \left( 1 - V(z) \right) \right]_0 \mod 2\mathcal{O}_2\\
&\hspace{-1em}\overset{\cref{eq: partial integration for framing coefficients}}\equiv 0 \mod 2\mathcal{O}_2
\end{align*}
\end{itemize}
This finishes the proof.
\end{proof}

\begin{theorem}[Analogue to \Cref{thm: thm 4.34} for fractional framing]\label{thm: fractional analogue of framed 3-functions}
Let $\rho, \sigma \in \mathbb{N}$ with $\gcd(\rho, \sigma)= 1$.
Then
\begin{align*}
\left( \frac1\sigma \varepsilon_\rho^{(3)} \circ \mathscr{C}_\sigma \circ \Phi^{+/-} \right) \left( \left( \frac\rho\sigma\mathbb{Z} \right) \times \mathcal{S}^2_\mathrm{rat} (K|\mathbb{Q}) \right) &\subset \overline{\mathcal{S}}^3(K|\mathbb{Q})_\mathrm{fin}.
\end{align*}
More precisely, for a rational $2$-function $V\in \mathcal{S}_{\mathrm{rat}}^2(K|\mathbb{Q})$ of periodicity $N$ and $\nu \in \frac\rho\sigma \mathbb{Z}$ and $S= \{p \text{ prim with } p\mid N\}\cup\{2,3\}$,
\begin{align*}
\widetilde V (z) := \frac1\sigma \varepsilon^{(3)}_\rho \left( \mathscr{C}_\sigma \left( \Phi^+(\nu, V) \right) \right) \in \mathcal{S}^3(K|\mathbb{Q})_S.
\end{align*}
For $\tilde a^+_n = \left[ \widetilde V(z)\right]_n$, $n\in \mathbb{N}$ we have
\begin{align*}
\mathrm{Frob}_\mathfrak{p} \left( \tilde a_n^+ \right) - \tilde a_{pn}^+ \equiv 0 \mod p^{2 \ord_p(pn) + \ord_p(\rho)-  \delta_{2,p} + \max\{ 0, \ord_p \left( \frac{pn}\rho \right) - \gamma_p \}} \mathcal{O}_\mathfrak{p},
\end{align*}
where $\gamma_p$ is equal to $1 + \ord_2(N+1)$, $1$ and $0$ if $p$ is equal to $2$, $3$ and greater than $3$, respectively.
In particular, for unramified $p\geq 5$ in $K|\mathbb{Q}$ with $p\nmid N$, and all $m,r\in \mathbb{N}$,
\begin{align*}
\frob_\mathfrak{p}\left( \tilde a_{mp^{r-1}}^+ \right) - \tilde a_{mp^r}^+ \equiv 0 \mod p^{3r}\mathcal{O}_\mathfrak{p}.
\end{align*}
\end{theorem}

\begin{proof}
The proof we are presenting here goes follows the same arguments and steps as the proof of \Cref{integrality of framing}.
In the following assume $\nu\neq 0$.
Indeed, $\Phi^+(0, V)= \Phi^-(0,V)= V$, as has been checked previously.

\noindent We write $\widetilde V (z)= \frac1\sigma \varepsilon^{(3)}_\rho \left( \mathscr{C}_\sigma \left( \Phi^+ ( \nu , V ) \right) \right)$ and $\tilde a_n^+ := \left[ \widetilde V(z)\right]_n$ for all $n \in \mathbb{N}$.
We have
\begin{align*}
\tilde a_n^+ =\frac1\sigma \left[ \varepsilon^{(3)}_\rho \left( \mathscr{C}_\sigma \left( \Phi^+ ( \nu, V ) \right) \right) \right]_n = \frac{\rho^3}\sigma \left[ \mathscr{C}_\sigma V^{(+,\nu)}(z) \right]_{\nicefrac n\rho}
=\frac{\rho^3}\sigma a_{\nicefrac{\sigma n}\rho}^+,
\end{align*}
with the understanding that $a_{\nicefrac{\sigma n}\rho}^+= 0$, whenever $\rho \nmid n$.
Then
\begin{align*}
\mathrm{Frob}_\mathfrak{p}(\tilde a^+_n) - \tilde a^+_{pn}=
\begin{cases}
0,& \text{if $\rho \nmid pn$},\\
-\frac{\rho^3}\sigma a^+_{\nicefrac{\sigma pn}\rho},& \text{if $\rho \mid pn$, but $\rho\nmid n$},\\
\frac{\rho^3}\sigma \left( \mathrm{Frob}_p\left( a_{\nicefrac{\sigma n}\rho}^+ \right) -  a_{\nicefrac{\sigma pn}\rho}^+ \right),& \text{if $\rho\mid n$}.
\end{cases}
\end{align*}
For $\rho\nmid pn$, and $\rho\mid pn$ but $\rho\nmid n$, the local 3-function property at the prime $p$ for the coefficients $a_n^+$ is trivially satisfied.
In the following, we will assume $\ord_p(\rho)\leq \ord_p(n)$.
\begin{itemize}
\item[\textit{Step 1:}] \textit{Analogue to \Cref{3-function characterization}}.
Here, we only assume $V\in \mathcal{S}^3(K|\mathbb{Q})$ and set $X(z)= \mathrm{Frob}_p V(z^p) - V(z)$.
We have
\begin{align*}
\mathrm{Frob}_p \left( a_{\nicefrac{\sigma n}\rho}^+ \right) - a_{\nicefrac{p \sigma n}\rho}^+
= \left[\frac{\exp\left( \nu \frac{\sigma n}\rho p \smallint V(z) \right)}{z^{\nicefrac{p\sigma n}\rho}} \cdot \sum_{k=1}^\infty (\nu \sigma)^{k-1} \frac\sigma{k!}\left(\frac{np}\rho\right)^k \left( \smallint X(z) \right)^k \right]_0.
\end{align*}
For $p\geq 3$ and $k\geq 4$
\begin{align*}
\ord_p\left( \frac\sigma{k!}\left(\frac{np}\rho\right)^k\right) &\geq 3 \left( \ord_p\left(\frac n\rho \right) +1 \right) + \ord_p(\sigma).
\end{align*}
For $p = 2$, $k\geq 4$ and $\ord_2\left( \nu \frac\sigma\rho n \right) > 0$, we have
\begin{align*}
\ord_2\left( (\nu \sigma)^{k-1} \frac\sigma{k!}\left( \frac{2n}\rho \right)^k \right) 
&> 3 \ord_2\left( \frac n\rho \right) + 2 + \ord_2(\sigma).
\end{align*}
For $k=3$ we still have
\begin{align*}
\ord_p \left( (\nu \sigma)^2 \frac\sigma{3!} \left( \frac{pn}\rho\right)^3 \right)
= \begin{cases}
3(\ord_p\left( \frac n\rho \right) + 1 ) + \ord_p(\sigma),& \text{for $p\geq 5$, and} \\
3\left( \ord_p\left( \frac n\rho \right) +1 \right) + \ord_p(\sigma) - 1,& \text{for $p\in\{2,3\}$}.
\end{cases}
\end{align*}
Therefore, analogously to \cref{eq: framed 3-function p>3} and \cref{eq: framed 3-function p<5}, we obtain
\begin{align*}
\frac\rho{np\sigma} \cdot \left( \mathrm{Frob}_\mathfrak{p} \left( a_{\nicefrac{\sigma n}\rho}^+ \right) - a_{\nicefrac{p\sigma n}\rho}^+ \right)
&\equiv\\
&\hspace{-13em}  \left[ \frac{ \exp \left( \nu \frac\sigma\rho n p \smallint V(z) \right) }{z^{\nicefrac{p\sigma n}\rho}} \left( \smallint X(z) + \frac{\nu np \sigma }{2\rho} (\smallint X(z))^2\right)
\right]_0\mod p^{2 \left( \ord_p \left( \frac n\rho \right)+1 \right) -\epsilon_p}\mathcal{O}_\mathfrak{p},
\end{align*}
where $\epsilon_p=0$ for all $p\geq 5$, and $\epsilon_p=1$ for $p \in \{ 2 , 3 \}$.
By the same calculation as in the proof of \Cref{3-function characterization} we obtain
\begin{align*}
\frac{2 \rho^2}{p^2 n^2 \sigma} \cdot \left( \mathrm{Frob}_\mathfrak{p} \left( a^+_{\nicefrac{\sigma n}\rho} \right) - a^+_{\nicefrac{pn \sigma}\rho} \right)
\equiv \nu \sigma \left[ \delta \left( \mathrm{Frob}_\mathfrak{p} V(z^p) + V(z) \right)\times\vphantom{ \left( \frac{\exp( \nu \smallint V(z) )} z\right)^{\nicefrac{pn\sigma}\rho} }\right.\\
&\hspace{-26em} \times\left.  \left( \frac{\exp( \nu \smallint V(z) )} z \right)^{\nicefrac{pn\sigma}\rho} \cdot
\smallint \hspace{-0.25em}\,^3 \left( \mathrm{Frob}_\mathfrak{p} V(z^p) - V(z)\right) \right]_0 \hspace{-1em}\mod p^{ \ord_p \left( \frac n\rho \right) + 1 - \delta_{3,p}} \mathcal{O}_\mathfrak{p}.
\end{align*}

\item[\textit{Step 2:}] \textit{Analogue to \Cref{cor:3-function characterization}.}
From now on, we will additionally assume $V\in \mathcal{S}^\infty (K|\mathbb{Q})$.
Then the same calculation (i.e. by the partial integration principle \cref{intbyparts}) as in the proof of \Cref{cor:3-function characterization} leads directly the $p$-adic estimation
\begin{align*}
\frac{2\rho^2}{p^2n^2\sigma}\cdot \left( \mathrm{Frob}_\mathfrak{p} \left( a^+_{\nicefrac{n\sigma}\rho} \right) - a^+_{\nicefrac{pn \sigma}\rho} \right) &\equiv \nu \sigma \left[ V(z) \cdot \left( \frac{\exp(\nu \smallint V(z))}{z}\right)^{\nicefrac{p n \sigma}\rho} \right. \times\\
&\hspace{-2em} \left. \vphantom{\left( \frac{\exp(\nu \smallint V(z))}{z}\right)^{\nicefrac{p n \sigma}\rho}}\smallint\hspace{-0.25em}\,^2 \left(\mathrm{Frob}_\mathfrak{p}V(z^p) - V(z)\right)\right]_0 \mod p^{\ord_p \left( \frac n\rho \right) + 1 - \delta_{3,p}} \mathcal{O}_\mathfrak{p}.
\end{align*}

\item[\textit{Step 3:}] \textit{Analogue to the proof of \Cref{thm: thm 4.34}.}
We now assume $V\in \mathcal{S}^2_{\mathrm{rat}}(K|\mathbb{Q})\subset \mathcal{S}^\infty ( K | \mathbb{Q})_\mathrm{fin}$.
More precisely, we have $V \in \mathcal{S}^\infty (K|\mathbb{Q})_S$.
In the following we assume $p\nmid N$.
By \Cref{lemma: a trivial lemma}, we immediately notice
\begin{align*}
\exp\left( \nu \frac{pn \sigma}\rho \smallint V(z) \right) \equiv \exp\left( \nu \frac{n \sigma}\rho \sum_{k=1}^\infty \frac{a_{pk}}k z^{pk} \right) \mod p^{\ord_p \left( \nu \frac\sigma\rho n \right)+1} \mathcal{O}_\mathfrak{p}.
\end{align*}
For $\widetilde Y(z):= \exp\left( \nu \frac{n \sigma}\rho \sum_{k=1}^\infty \frac{a_{pk}}k z^{pk} \right)$ we have by \Cref{lemma: final countdown}, and since $\ord_p\left(\nu \frac\sigma\rho \right)\geq 0$,
\begin{align*}
\tilde y_m := \left[\tilde Y(z)\right]_m \equiv 0 \mod p^{\max \left\{ 0 , \ord_p \left( \nu \frac\sigma\rho n \right)  -\ord_p(m) \right\}}\mathcal{O}_\mathfrak{p}.
\end{align*}
By the previous two steps and some calculation
\begin{align}\label{eq: fractional fraction final estimation}
\frac{2\rho^2}{p^2n^2 \sigma}\cdot \left( \mathrm{Frob}_\mathfrak{p} \left( a^+_{\nicefrac{n\sigma}\rho} \right)- a^+_{\nicefrac{pn \sigma}\rho} \right) & \nonumber\\
&\hspace{-7em}\equiv \nu \sigma \sum_{m=0}^{\nicefrac{\sigma n}\rho} \tilde y_m \sum_{\substack{\ell = 1 \\ p\nmid \ell}}^{p(\frac{\sigma n}\rho - m)} \frac{a_{p(\frac{\sigma n}\rho - m)-\ell} a_\ell}{\ell^2} \mod p^{\ord_p\left( \frac n\rho \right) + 1 - \delta_{3,p}} \mathcal{O}_\mathfrak{p}.
\end{align}
We need to give an estimation for $x(m)$, $m=0, ... , \frac{\sigma n}\rho$, defined by
\begin{align}\label{eq: countdown}
x(m)= \ord_p \left( \nu \sigma \tilde y_m \sum_{\substack{\ell = 1 \\ p\nmid \ell}}^{p(\frac{\sigma n}\rho - m)} \frac{a_{p(\frac{\sigma n}\rho - m)-\ell} a_\ell}{\ell^2} \right).
\end{align}
For $m=0$ we have $\tilde y_0=1$ and by \Cref{lemma: general wolstenholme},
\begin{align*}
x(0)&\geq 
\min \big\{ \ord_p ( n ) - \ord_p ( \rho ) + 1 - \delta_{ 3 , p }, \max \{ 0, \ord_p ( n ) - \ord_p ( \rho ) +1 - \gamma_p \} \big\}\\
&=\max \left\{ 0, \ord_p\left( \frac n\rho \right) +1 - \gamma_p \right\}.
\end{align*}
Therefore, we may assume $m> 0$ in the following.
By \Cref{lemma: general wolstenholme} and \Cref{lemma: final countdown} we obtain (for $m>0$)
\begin{align*}
x(m)&\geq \min \left\{ \ord_p \left(\frac{pn}\rho \right) - \delta_{3,p}, \ord_p \left( \nu \sigma \right) + \max \left\{ 0, \ord_p\left( \frac{\nu n \sigma}{\rho m} \right) \right\} \right.\\
&\hspace{10em} \left. \vphantom{\left(\frac p\rho\right)}+ \max \left\{ 0, \ord_p \left( \frac{\sigma n}\rho - m \right) +1 - \gamma_p \right\} \right\} ,
\end{align*}
Taking into account, that $\ord_p(\nu \sigma)\geq 0$ for all unramified primes $p$, we immediately get
\begin{align*}
x= x(m)&\geq \min \left\{ \ord_p \left(\frac{pn}\rho \right) - \delta_{3,p},  \max \left\{ 0, \ord_p\left( \frac n{\rho m} \right) \right\}  \right.\\
&\hspace{5em}\left. \vphantom{\left(\frac p\rho\right)}+ \max \left\{ 0, \ord_p \left( \frac{\sigma n}\rho - m \right) +1 - \gamma_p \right\} \right\}.
\end{align*}
Recall that we assume $\rho\mid n$ and $\gcd(\sigma,\rho)=1$. In particular this means $\ord_p(\rho)\leq \ord_p(n)$ for all primes $p$.
\begin{itemize}

\item[•]
If $\ord_p\left( \frac n{\rho m}\right)\geq 0$ and $\ord_p\left( \frac{\sigma n}\rho - m \right) + 1 \geq \gamma_p$, then $\ord_p\left( \frac{\sigma n}\rho - m \right) = \ord_p(m)$ and hence
\begin{align*}
x&\geq \min\left\{ \ord_p\left( \frac{pn}\rho \right) - \delta_{3,p}, \ord_p\left( \frac n{\rho m} \right) + \ord_p(m) +1 -\gamma_p \right\}\\
&= \ord_p \left( \frac{pn}\rho\right) - \gamma_p\geq 0.
\end{align*}

\item[•]
If $\ord_p\left( \frac n{\rho m}\right)\geq 0$ and $\ord_p\left( \frac{\sigma n}\rho - m \right) + 1 < \varepsilon_{p,N}$, then $\ord_p\left( \frac{\sigma n}\rho - m \right) = \ord_p(m)$ and $-\ord_p(m) > 1- \gamma_p$.
Hence
\begin{align*}
x&\geq \min\left\{ \ord_p\left( \frac{pn}\rho \right) - \delta_{3,p}, \ord_p\left( \frac n\rho \right) - \ord_p(m) \right\}\\
&\geq \max\left\{ 0 , \min \left\{ \ord_p\left( \frac{pn}\rho \right) - \delta_{3,p}, \ord_p\left(\frac n\rho \right) + 1 - \gamma_p \right\}\right\}\\
&=\max\left\{ 0 , \ord_p\left(\frac {pn}\rho \right) - \gamma_p \right\}.
\end{align*}

\item[•]
If $\ord_p\left( \frac n{\rho m}\right) < 0$ and $\ord_p\left( \frac{\sigma n}\rho - m \right) + 1 \geq \gamma_p$, then, in particular,
\begin{align*}
\min\left\{ \ord_p\left( \frac{\sigma n}\rho \right) , \ord_p(m) \right\} \geq \ord_p \left( \frac n\rho \right).
\end{align*}
Hence,
\begin{align*}
x&\geq \min\left\{ \ord_p\left( \frac{pn}\rho \right) - \delta_{3,p}, \ord_p\left( \frac n{\rho m} \right) +1 -\gamma_p \right\}\\
&\geq \max\left\{ 0 , \min\left\{ \ord_p\left( \frac{pn}\rho \right) - \delta_{3,p}, \right.\right. \\
&\hspace{14em} \left.\left. \min\left\{ \ord_p\left( \frac{\sigma n}\rho \right) , \ord_p(m) \right\} + 1 - \gamma_p \right\} \right\}\\
&\geq \max \left\{0,  \min\left\{ \ord_p\left( \frac{pn}\rho \right) - \delta_{3,p}, \ord_p\left( \frac n\rho \right) + 1 - \gamma_p \right\}\right\}\\
&\geq  \max \left\{ 0 , \ord_p \left( \frac{pn}\rho \right) - \gamma_p \right\}.
\end{align*}

\item
If $\ord_p\left( \frac n{\rho m}\right) < 0$ and $\ord_p\left( \frac{\sigma n}\rho - m \right) + 1 < \gamma_p$, then $ x \geq 0$.
On the other hand,
\begin{align*}
\ord_p\left( \frac{pn}\rho \right) - \gamma_p &\leq \min\left\{ \ord_p\left( \frac{\sigma n}\rho \right), \ord_p(m) \right\} + 1 - \gamma_p\\
&\leq \ord_p\left( \frac{\sigma n}\rho -m \right) + 1 - \gamma_p < 0.
\end{align*}

\end{itemize}

Summarizing the above considerations, we obtain
\begin{align*}
\min_{m\in \{0, ..., \nicefrac{\sigma n}\rho\}} x(m) \geq \max\left\{ 0, \ord_p\left(\frac{pn}\rho\right) - \gamma_p \right\}.
\end{align*}
Therefore,
\begin{align*}
\nu\sigma \sum_{m=0}^{\nicefrac{\sigma n}\rho} \tilde y_m \sum_{\substack{\ell = 1 \\ p\nmid \ell}}^{p(\frac{\sigma n}\rho - m)} \frac{a_{p(\frac{\sigma n}\rho - m)-\ell} a_\ell}{\ell^2}
\equiv 0 \mod p^{\max\{ 0 , \ord_p \left( \frac{pn}\rho \right) - \gamma_p \} } \mathcal{O}_\mathfrak{p}.
\end{align*}
Consequently, by \cref{eq: fractional fraction final estimation}, we obtain -- except for the case $p=2$ and $\ord_2 \left( \nu \frac\sigma\rho n \right) = 0$ --
\begin{align*}
\mathrm{Frob}_\mathfrak{p} \left( a^+_{\nicefrac{n\sigma}\rho} \right)- a^+_{\nicefrac{pn \sigma}\rho} \equiv 0 \mod p^{ 2 \ord_p\left( \frac{pn}\rho \right) - \delta_{2,p} + \ord_p(\sigma) + \max \left\{ 0, \ord_p\left( \frac{pn}\rho \right) - \gamma_p \right\}} \mathcal{O}_\mathfrak{p}.
\end{align*}
\end{itemize}
This finishes the proof.
\end{proof}

\begin{example}[Jacobsthal-Kazandzidis] \label{ex: jacobsthal-kazandzidis}
Let $V(z)= \frac{z}{1-z}\in \mathcal{S}^2_{\mathrm{rat}}(\mathbb{Q})$ and $\nu = \frac \rho\sigma$ with $\rho,\sigma\in\mathbb{N}$, $\gcd(\rho,\sigma)=1$.
$V$ has periodicity $N=1$ and
\begin{align*}
\smallint V(z)= \smallint \left( \frac z{1-z} \right) = \sum_{k=1}^\infty \frac{z^k}k = - \log(1-z).
\end{align*}
As always, $a_n^+:= \left[\Phi^+(\nu, V) \right]_n$
Recall from \cref{eq: framed coefficients}
\begin{align*}
a_n^+ = \frac1\nu\left[\frac{\exp(\nu n \smallint V(z))}{z^n}\right]_0,\quad
\text{for all $n\in\mathbb{N}$}.
\end{align*}
Then
\begin{align*}
a_n^+= \frac1\nu\left[ \frac1{z^n} ( 1 - z )^{- \nu n} \right]_0.
\end{align*}
By the \textit{generalized Binomial Theorem} we have
\begin{align*}
(1-z)^{-\nu n} =\sum_{k=0}^\infty\binom{-\nu n}k (-1)^k z^k.
\end{align*}
Note that in this case the binomial coefficient is defined by
\begin{align}\label{eq: general binomial}
\binom{-\nu n}k = \frac1{k!} \prod_{j=0}^{k-1} (- \nu n -j).
\end{align}
Rewriting the binomial coefficient, we obtain
\begin{align*}
\binom{-\nu n }k &= \frac1{k!} \prod_{j=0}^{k-1} ( \nu n + k - 1 -  j )
= \frac{(-1)^k}{k!} \prod_{j=1}^k ( \nu n + k - j )\\
&= \frac{(-1)^k}{k!} \frac{\nu n}{\nu n + k} \prod_{j=0}^{k-1} ( \nu n + k - j )
= (-1)^k \frac{\nu n}{\nu n + k } \binom{\nu n + k}k.
\end{align*}
Therefore,
\begin{align*}
a_n^+ = \frac1\nu \left[ \frac1{z^n} \sum_{k=0}^\infty \frac{\nu n}{\nu n + k} \binom{\nu n + k}k z^k \right]_0
= \frac1{\nu + 1} \binom{(\nu + 1) n}n.
\end{align*}
In particular, for $\widetilde V(z) = \frac1\sigma \mathscr{C}_\sigma \left( \Phi^+(\nu, V) \right)$,
\begin{align*}
a^+_{\sigma n} = \left[ \widetilde V(z) \right]_n = \frac1{\rho+\sigma} \binom{ ( \rho + \sigma ) n }{\sigma n }.
\end{align*}
Applying \Cref{thm: fractional analogue of framed 3-functions} to  $\widetilde V(z) $ gives for all primes $p\geq 3$
\begin{align*}
  a_{\sigma n}^+ - a_{\sigma pn}^+ \equiv 0 \mod p^{3(\ord_p(n)+1)-\delta_{p,3}}\mathbb{Z}_p.
\end{align*}
On the other hand for $p\geq 3$
\begin{align*}
a_{\sigma n}^+ - a_{\sigma pn}^+ &= \frac1{\rho+\sigma} \left[ \binom{(\rho+\sigma)n}{\sigma n} - \binom{(\rho+\sigma)pn}{\sigma pn} \right].
\end{align*}
Therefore,
\begin{align*}
\binom{(\rho + \sigma)n}{\sigma n} - \binom{(\rho + \sigma )pn}{\sigma pn} &= ( \rho + \sigma ) \left( a_{\sigma n}^+ - a_{\sigma pn}^+ \right)\\
&\equiv 0 \mod p^{3(\ord_p(n)+1)- \delta_{p,3} + \ord_p(\rho + \sigma )}.
\end{align*}
Equivalently,
\begin{align}\label{eq: jacobsthal-kazandzidis from fractional framing}
\binom{(\rho + \sigma )pn}{\sigma pn} \equiv \binom{(\rho + \sigma )n}{\sigma n} \mod p^{3(\ord_p(n)+1)- \delta_{p,3} + \ord_p(\rho + \sigma )}.
\end{align}
Now we will prove the Theorem of Jacobsthal-Kazandzidis (see \Cref{thm: jacobsthal-kazandzidis}) for $p\geq 3$ as a consequence of \Cref{thm: fractional analogue of framed 3-functions}.
Fix a prime $p\geq 3$, let $a,b\in \mathbb{N}_0$ be non-negative integers and let $r\in \mathbb{N}$ be an integer.
W. l. o. g.,  let $b\leq a$ and $\gamma = \min\{\ord_p(a), \ord_p(b)\}$.
Then either $bp^{-\gamma}$ or $(a-b)p^{-\gamma}$ is a $p$-adic unit.
Since the binomial coefficient is symmetric (that is, $\binom ab$ invariant under the exchange $b \leftrightarrow a-b$), we may assume that $bp^{-\gamma}$ is a $p$-adic integer.
Then \Cref{thm: jacobsthal-kazandzidis} follows from \cref{eq: jacobsthal-kazandzidis from fractional framing} by setting $\sigma =p^{-\gamma}b$, $\rho = p^{-\gamma}(a-b)$ and $n=p^{\gamma+r-1}$, i.e.
\begin{align*}
\binom{a p^r}{b p^{r}}  &\equiv \binom{ap^{r-1}}{bp^{r-1}}  \mod p^{3(r+\gamma) -\delta_{p,3}}.
\end{align*}
\end{example}


\bibliographystyle{amsplain}

\end{document}